\documentclass[10pt]{article}
 \usepackage{amsmath,amsfonts,amssymb,amsthm}
 \usepackage{hyperref}
 \usepackage{slashed}

 \theoremstyle{plain}
 \newtheorem{theorem}{Theorem}[section]
 \newtheorem{lemma}[theorem]{Lemma}
 \newtheorem{corollary}[theorem]{Corollary}
 \newtheorem{proposition}[theorem]{Proposition}
 
 \newtheorem{remark}[theorem]{Remark}

 \theoremstyle{definition}
 \newtheorem{definition}[theorem]{Definition}

 \title{On J-equation}

 \author{Gao Chen}

\begin{document}

 \maketitle
\begin{abstract}
   In this paper, we prove that for any K\"ahler metrics $\omega_0$ and $\chi$ on $M$, there exists $\omega_\varphi=\omega_0+\sqrt{-1}\partial\bar\partial\varphi>0$ satisfying the J-equation $\mathrm{tr}_{\omega_\varphi}\chi=c$ if and only if $(M,[\omega_0],[\chi])$ is uniformly J-stable. As a corollary, we can find many constant scalar curvature K\"ahler metrics with $c_1<0$. Using the same method, we also prove a similar result for the deformed Hermitian-Yang-Mills equation when the angle is in $(\frac{n\pi}{2}-\frac{\pi}{4},\frac{n\pi}{2})$.
\end{abstract}
\section{Introduction}
  In this paper, our main goal is to prove the equivalence of the solvability of the J-equation and a notion of stability. Given K\"ahler metrics $\omega_0$ and $\chi$ on $M$, the J-equation is defined as \[\mathrm{tr}_{\omega_\varphi}\chi=c\] for \[\omega_\varphi=\omega_0+\sqrt{-1}\partial\bar\partial\varphi>0.\]

 In general, the equivalence of the stability and the solvability of an equation is very common in geometry. One of the first results in this direction was the celebrated work of Donaldson-Uhlenbeck-Yau \cite{DonaldsonGauge, UhlenbeckYau} on Hermitian-Yang-Mills connections. Inspired by the study of Hermitian-Yang-Mills connections, Donaldson proposed many questions including the study of J-equation using the moment map interpretation \cite{DonaldsonJEquation}. It was the first appearance of the J-equation in the literature.

Yau conjectured that the existence of Fano K\"ahler-Einstein metric is also equivalent to some kind of stability \cite{YauStability}. Tian made it precise in Fano K\"ahler-Einstein case and it was called the K-stability condition \cite{Tian}. It was generalized by Donaldson to the constant scalar curvature K\"ahler (cscK) problem in projective case \cite{Donaldson}. This conjecture has been proved by Chen-Donaldson-Sun \cite{ChenDonaldsonSun1, ChenDonaldsonSun2, ChenDonaldsonSun3} in Fano K\"ahler-Einstein case. However, there is evidence that this conjecture may be wrong in cscK case \cite{ApostolovCalderbankGauduchonTonnesen}. There is a folklore conjecture that the uniform version of K-stability may be a correct substitution. When restricted to special test configurations called ``degeneration to normal cones", the uniform K-stability is reduced to Ross-Thomas's uniform slope K-stability \cite{RossThomas}. More recently, the projective assumption was removed by the work of Dervan-Ross \cite{DervanRoss} and independently by Sj\"ostr\"om Dyrefelt \cite{SjostromDyrefelt}.

It is easy to see that cscK metrics are critical points of the K-energy functional \cite{Chen}
\[K(\varphi)=\int_M\log(\frac{\omega_\varphi^n}{\omega_0^n})\frac{\omega_\varphi^n}{n!}+\mathcal{J}_{-\mathrm{Ric}(\omega_0)}(\varphi).\]
The $\mathcal{J}_\chi$  functional for any (1,1)-form $\chi$ is defined by
\[\mathcal{J}_\chi(\varphi)=\frac{1}{n!}\int_M\varphi\sum_{k=0}^{n-1}\chi\wedge\omega_0^k\wedge\omega_\varphi^{n-1-k}-\frac{1}{(n+1)!}\int_M c_0\varphi\sum_{k=0}^{n}\omega_0^k\wedge\omega_\varphi^{n-k},\]
where $c_0$ is the constant given by \[\int_M\chi\wedge\frac{\omega_0^{n-1}}{(n-1)!}-c_0\frac{\omega_0^n}{n!}=0.\] When $\chi$ is a K\"ahler form, it is well known that the critical point of the $\mathcal{J}_\chi$  functional is exactly the solution to the J-equation. It was the second appearance of the J-equation in the literature. Following this formula, using the interpolation of the K-energy and the $\mathcal{J}_\chi$ functional, Chen-Cheng \cite{ChenCheng1, ChenCheng2, ChenCheng3} proved that the existence of cscK metric is equivalent to the geodesic stability of K-energy. However, the relationship between the existence of cscK metrics and the uniform K-stability is still open.

When we replace the K-energy by the $\mathcal{J}_\chi$ functional for a K\"ahler form $\chi$, the analogy of the K-stability and the slope stability conditions were proposed by Lejmi and Sz\'ekelyhidi \cite{LejmiSzekelyhidi}. See also Section 6 of \cite{DervanRoss} for the extension to non-projective case. The main theorem of this paper proves the equivalence between the existence of the critical point of $\mathcal{J}_\chi$ functional, the solvability of J-equation, the coerciveness of $\mathcal{J}_\chi$ functional, and the uniform J-stability as well as the uniform slope J-stability.
\begin{theorem}
(Main Theorem)
Fix a K\"ahler manifold $M^n$ with K\"ahler metrics $\chi$ and $\omega_0$. Let $c_0>0$ be the constant such that
\[\int_M \chi\wedge\frac{\omega_0^{n-1}}{(n-1)!}=c_0\int_M\frac{\omega_0^n}{n!},\]
then the followings are equivalent:

(1) There exists a unique smooth function $\varphi$ up to a constant such that $\omega_\varphi=\omega_0+\sqrt{-1}\partial\bar\partial\varphi>0$ satisfies the J-equation
\[\mathrm{tr}_{\omega_\varphi}\chi=c_0;\]

(2) There exists a unique smooth function $\varphi$ up to a constant such that $\omega_\varphi=\omega_0+\sqrt{-1}\partial\bar\partial\varphi>0$ satisfies the J-equation
\[\chi\wedge\frac{\omega_\varphi^{n-1}}{(n-1)!}=c_0\frac{\omega_\varphi^n}{n!};\]

(3) There exists a unique smooth function $\varphi$ up to a constant such that $\varphi$ is the critical point of the $\mathcal{J}_\chi$ functional;

(4) The $\mathcal{J}_\chi$ functional is coercive, in other words, there exist $\epsilon_{1.1}>0$ and another constant $C_{1.2}$ such that $\mathcal{J}_\chi(\varphi)\ge \epsilon_{1.1}\mathcal{J}_{\omega_0}(\varphi)-C_{1.2}$;

(5) $(M,[\omega_0],[\chi])$ is uniformly J-stable, in other words, there exists $\epsilon_{1.1}>0$ such that for all K\"ahler test configurations $(\mathcal{X},\Omega)$ defined as Definition 2.10 of \cite{DervanRoss}, the numerical invariant $J_{[\chi]}(\mathcal{X},\Omega)$ defined as Definition 6.3 of \cite{DervanRoss} satisfies $J_{[\chi]}(\mathcal{X},\Omega)\ge\epsilon_{1.1}J_{[\omega_0]}(\mathcal{X},\Omega)$;

(6) $(M,[\omega_0],[\chi])$ is uniformly slope J-stable, in other words, there exists $\epsilon_{1.1}>0$ such that for any subvariety $V$ of $M$, the degeneration to normal cone $(\mathcal{X},\Omega)$ defined as Example 2.11 (ii) of \cite{DervanRoss} satisfies $J_{[\chi]}(\mathcal{X},\Omega)\ge\epsilon_{1.1}J_{[\omega_0]}(\mathcal{X},\Omega)$;

(7) There exists $\epsilon_{1.1}>0$ such that
\[\int_V (c_0-(n-p)\epsilon_{1.1})\omega_0^p-p\chi\wedge\omega_0^{p-1}\ge0\] for all $p$-dimensional subvarieties $V$ with $p=1,2,...,n$.

\label{Main-theorem}
\end{theorem}
\begin{remark}
Lejmi and Sz\'ekelyhidi's original conjecture is that the solvability of \[\mathrm{tr}_{\omega_\varphi}\chi=c_0\] is equivalent to \[\int_V c_0\omega_0^p-p\chi\wedge\omega_0^{p-1}>0\] for all $p$-dimensional subvarieties $V$ with $p=1,2,...,n-1$ \cite{LejmiSzekelyhidi}. However, it seems that our uniform version is more natural from geometric point of view.
\end{remark}
\begin{remark}
It is well known that there exists a constant $C_{1.3}$ depending on $n$ such that the $\mathcal{J}_{\omega_0}$ functional
\[\int_0^1(\int_M\varphi(\omega_0\wedge\frac{\omega_{t\varphi}^{n-1}}{(n-1)!}-n\frac{\omega_{t\varphi}^n}{n!}))dt=\int_0^1(\sqrt{-1}\int_M\partial\varphi\wedge\bar\partial\varphi\wedge\frac{t\omega_{t\varphi}^{n-1}}{(n-1)!})dt\]
and Aubin's I-functional \[\int_M\varphi(\omega_0^n-\omega_\varphi^n)
=\sqrt{-1}\int_M\partial\varphi\wedge\bar\partial\varphi\wedge\sum_{k=0}^{n}\omega_0^k\wedge\omega_\varphi^{n-k}\]
satisfy \[C_{1.3}^{-1}\int_M\varphi(\omega_0^n-\omega_\varphi^n)\le\mathcal{J}_{\omega_0}(\varphi)\le C_{1.3}\int_M\varphi(\omega_0^n-\omega_\varphi^n).\]
For example,  Collins and Sz\'ekelyhidi used this fact and their Definition 20 in \cite{CollinsSzekelyhidi} replaced $\mathcal{J}_{\omega_0}(\varphi)$ by $\int_M\varphi(\omega_0^n-\omega_\varphi^n)$ in the definition of the coerciveness which was called ``properness" in \cite{CollinsSzekelyhidi}. By (3) of \cite{Aubin}, Aubin's I-functional can also be replaced by Aubin's J-functional in the definition of coerciveness. Accordingly, in the definition of uniform stability, the numerical invariant $J_{[\omega_0]}(\mathcal{X},\Omega)$ can be replaced by the minimum norm of $(\mathcal{X},\Omega)$ defined as Definition 2.18 of \cite{DervanRoss}. By (62) of \cite{Darvas}, Aubin's J-functional can be further replaced by the $d_1$ distance in the definition of the coerciveness when $\varphi$ is normalized such that the Aubin-Mabuchi energy of $\varphi$ is 0.
\label{J-I-equivalence}
\end{remark}
\begin{remark}
By Proposition 2 of \cite{Chen}, if the solution to the J-equation exists, it is unique up to a constant. It is easy to see that (1) and (2) are equivalent. The equivalence between (2) and (3) follows from the formula
\[\frac{d\mathcal{J}_\chi}{dt}=\int_M\frac{\partial\varphi}{\partial t}(\chi\wedge\frac{\omega_\varphi^{n-1}}{(n-1)!}-c_0\frac{\omega_\varphi^n}{n!}).\]
By Proposition 21 and Proposition 22 of \cite{CollinsSzekelyhidi} and Remark \ref{J-I-equivalence}, (1) and (4) are equivalent. By Corollary 6.5 of \cite{DervanRoss}, (4) implies (5). It is trivial that (5) implies (6). By \cite{LejmiSzekelyhidi}, (6) implies (7) in the projective case if $\epsilon_{1.1}$ is replaced by 0. However, it is easy to see that it is also true in non-projective case and for positive $\epsilon_{1.1}$. Thus, we only need to prove that (7) implies (1) in Theorem \ref{Main-theorem}. Remark that there is a simpler proof that (1) implies (7). Let $\chi=\delta_{ij}$ and $\omega_\varphi=\lambda_i\delta_{ij}$, then for any $c>0$, the condition \[c\omega_\varphi^p-p\chi\wedge\omega_\varphi^{p-1}\ge0\] is equivalent to \[\sum_{j=1}^p\frac{1}{\lambda_{i_j}}\le c\] for all distinct $p$ numbers $\{i_1, i_2, ..., i_p\}\subset\{1, 2, ..., n\}$. So $\mathrm{tr}_{\omega_\varphi}\chi=c_0$ as well as the upper bounds of $\lambda_i$ imply that for small enough $\epsilon_{1.1}>0$, \[(c_0-(n-p)\epsilon_{1.1})\omega_\varphi^p-p\chi\wedge\omega_\varphi^{p-1}\ge0\] for all $p=1, 2, ..., n$. (4) follows from the fact that \[\int_V (c_0-(n-p)\epsilon_{1.1})\omega_0^p-p\chi\wedge\omega_0^{p-1}=\int_V (c_0-(n-p)\epsilon_{1.1})\omega_\varphi^p-p\chi\wedge\omega_\varphi^{p-1}.\]
\label{J-equation-implies-stability}
\end{remark}

When $c_1(M)<0$, we can choose $\chi$ as a K\"ahler form in $-c_1(M)$. Since $x\log x$ is bounded from below for any $x\in\mathbb{R}$, the entropy $\int_M\log(\frac{\omega_\varphi^n}{\omega_0^n})\frac{\omega_\varphi^n}{n!}$ is also bounded from below. So the coerciveness of $\mathcal{J}_\chi$ functional implies the coerciveness of K-energy. This observation appeared as Remark 2 of \cite{Chen}. Using this observation, as a corollary of Theorem 1.3 of \cite{ChenCheng2} and Theorem \ref{Main-theorem}, we can find many cscK metrics with $c_1(M)<0$.
\begin{corollary}
If $c_1(M)<0$, and $\epsilon_{1.1}>0$, then for any K\"ahler class $[\omega_0]$ such that \[\int_V ((\frac{-n[c_1(M)]\cdot[\omega_0]^{n-1}}{[\omega_0]^n}-(n-p)\epsilon_{1.1})\omega_0^p-p\omega_0^{p-1}\wedge(-c_1(M))\ge0\] for all $p$-dimensional subvarieties $V$ with $p=1,2,...,n$, there exists a cscK metric in $[\omega_0]$.
\label{cscK}
\end{corollary}
\begin{remark}
If there exists $\omega_\varphi\in[\omega_0]$ such that $\mathrm{Ric}(\omega_\varphi)<0$ and $\omega_\varphi$ has constant scalar curvature, then the condition above is also necessary.
\end{remark}

Besides the appearances in the moment map picture and the study of the cscK problem, J-equation also arises from the study of mirror symmetry. In fact, using the following observation of Collins-Jacob-Yau \cite{CollinsJacobYau} \[\lim_{k\rightarrow\infty}\sum_{i=1}^{n}k(\frac{\pi}{2}-\arctan(k\lambda_i))=\sum_{i=1}^{n}\frac{1}{\lambda_i},\] the J-equation is exactly the limit of the deformed Hermitian-Yang-Mills equation \[\sum_{i=1}^{n}\arctan{\lambda_i}=\text{constant},\] where $\lambda_i$ are the eigenvalues of $\omega_\varphi$ with respect to $\chi$. It plays an important role in the study of mirror symmetry \cite{StromingerYauZaslow, LeungYauZaslow}.

Motivated by the J-equation, Collins-Jacob-Yau \cite{CollinsJacobYau} conjectured that the solvability of the deformed Hermitian-Yang-Mills equation is also equivalent to a notion of stability. In this paper, we prove the uniform version of their conjecture when the angle is in $(\frac{n\pi}{2}-\frac{\pi}{4},\frac{n\pi}{2})$:

\begin{theorem}
Fix a K\"ahler manifold $M^n$ with K\"ahler metrics $\chi$ and $\omega_0$. Let $\hat\theta\in(\frac{n\pi}{2}-\frac{\pi}{4},\frac{n\pi}{2})$ be a constant. Then there exists a unique smooth function $\varphi$ up to a constant such that \[\sum_{i=1}^{n}\arctan \lambda_i=\hat\theta\] for eigenvalues $\lambda_i$ of $\omega_\varphi=\omega_0+\sqrt{-1}\partial\bar\partial\varphi>0$ with respect to $\chi$ if and only if there exists a constant $\epsilon_{1.1}>0$ and for all $p$-dimensional subvarieties $V$ with $p=1,2,...,n$ ($V$ can be chosen as $M$), there exist smooth functions $\theta_V(t)$ from $[1,\infty)$ to $[\hat\theta-\frac{(n-p)\pi}{2}+(n-p)\epsilon_{1.1},\frac{p\pi}{2})$ such that for all $t\in[1,\infty)$,
\[\int_V(\chi+\sqrt{-1}t\omega_0)^p\not=0, \theta_V(t)=\mathrm{arg}(\int_V(\chi+\sqrt{-1}t\omega_0)^p), \lim_{t\rightarrow\infty}\theta_V(t)=\frac{p\pi}{2}.\]
Moreover, when $V=M$, it is required that $\theta_M(1)=\hat\theta$.
\label{Deformed-Hermitian-Yang-Mills}
\end{theorem}

\begin{remark}
We only study the case when $\hat\theta\in(\frac{n\pi}{2}-\frac{\pi}{4},\frac{n\pi}{2})$ in this paper. So it is natural to assume that $[\omega_0]$ is a K\"ahler class. However, usually we need extra conditions in addition to $[\omega_0]$ being K\"ahler to make sure $\int_V(\chi+\sqrt{-1}t\omega_0)^p$ is not 0 so that $\theta_V(t)$ is well defined. When $p=1, 2$, $\theta_V(t)$ is always well defined and increasing for $t\in(-\infty,\infty)$ without any extra assumption. In addition, $\theta_V(0)=0$. When $p=3$, \[\int_V(\chi+\sqrt{-1}t\omega_0)^3=\int_V\chi^3-3t^2\int_V\chi\wedge\omega_0^2+\sqrt{-1}t(3\int_V\chi^2\wedge\omega_0-t^2\int_V\omega_0^3).\]
So if the inequality \[(\int_V\chi^3)(\int_V\omega_0^3)<9(\int_V\chi\wedge\omega_0^2)(\int_V\chi^2\wedge\omega_0)\] in Proposition 3.3 of \cite{CollinsXieYau} holds, then $\theta_V(t)$ is well defined for $t\in(-\infty,\infty)$. Moreover, if $\theta_V(1)>\pi$, then $\theta_V(t)$ is increasing for $t\in[1,\infty)$. In addition, $\theta_V(0)=0$. So the choice of $\theta_V(1)$ in this paper is the same as the choice of $\theta_V(1)$ in Proposition 8.4 of \cite{CollinsYau}. In higher dimensions, more inequalities are involved.
\end{remark}
\begin{remark}
Collins-Jacob-Yau conjectured that $\theta_V(1)>\hat\theta-\frac{(n-p)\pi}{2}$ for all $p=1,2,...,n-1$ is equivalent to the solvability of the deformed Hermitian-Yang-Mills equation \cite{CollinsJacobYau}. However, it seems that our uniform version is more natural because Definition 8.10 (2) of \cite{CollinsYau} also assumed the uniform positive lower bound.
\end{remark}
\begin{remark}
By Theorem 1.1 of \cite{JacobYau}, the solution to the deformed Hermitian-Yang-Mills equation is unique up to a constant if it exists. The ``only if" part of Theorem \ref{Deformed-Hermitian-Yang-Mills} is a combination of Proposition 3.1 of \cite{CollinsJacobYau} and Remark \ref{J-equation-implies-stability}. So we only need to prove the ``if" part of Theorem \ref{Deformed-Hermitian-Yang-Mills}.
\end{remark}

Theorem \ref{Deformed-Hermitian-Yang-Mills} will be proved in Section 5 using the same method of the proof of Theorem \ref{Main-theorem}.

Instead of Theorem \ref{Main-theorem}, we will prove the following stronger statement by induction:

\begin{theorem}
Fix a K\"ahler manifold $M^n$ with K\"ahler metrics $\chi$ and $\omega_0$. Let $c>0$ be a constant and $f>-\frac{1}{2n}(\frac{1}{c})^{n-1}$ be a smooth function satisfying
\[\int_M f\frac{\chi^n}{n!}=c\int_M\frac{\omega_0^n}{n!}-\int_M \chi\wedge\frac{\omega_0^{n-1}}{(n-1)!}\ge 0,\]
then there exists $\omega_\varphi=\omega_0+\sqrt{-1}\partial\bar\partial\varphi>0$ satisfying the equation
\[\mathrm{tr}_{\omega_\varphi}\chi+f\frac{\chi^n}{\omega_\varphi^n}=c\] and the inequality \[c\omega_\varphi^{n-1}-(n-1)\chi\wedge\omega_\varphi^{n-2}>0\]
if there exists $\epsilon_{1.1}>0$ such that \[\int_V (c-(n-p)\epsilon_{1.1})\omega_0^p-p\omega_0^{p-1}\wedge\chi\ge0\] for all $p$-dimensional subvarieties $V$ with $p=1,2,...,n$.
\label{Main-induction}
\end{theorem}

\begin{remark}
By Remark \ref{J-equation-implies-stability}, Theorem \ref{Main-theorem} is a corollary of Theorem \ref{Main-induction} by choosing $f=0$.
\end{remark}

\begin{remark}
When $n=1$, Theorem \ref{Main-induction} is trivial. When $n=2$, Theorem \ref{Main-induction} is the statement that the Demailly-Paun's characterization \cite{DemaillyPaun} for $[c\omega_0-\chi]$ being K\"ahler implies the solvability of the Calabi conjecture
\[(c\omega_\varphi-\chi)^2=(cf+1)\chi^2\] by Yau \cite{Yau}. In the toric case when $f$ is a non-negative constant, Theorem \ref{Main-induction} was proved by Collins and Sz\'ekelyhidi \cite{CollinsSzekelyhidi}.
\label{Base-case}
\end{remark}

There are several steps to prove Theorem \ref{Main-induction}.

Step 1: Prove the following:

\begin{theorem}
Fix a K\"ahler manifold $M^n$ with K\"ahler metrics $\chi$ and $\omega_0$. Let $c>0$ be a constant and $f>-\frac{1}{2n}(\frac{1}{c})^{n-1}$ be a smooth function satisfying
\[\int_M f\frac{\chi^n}{n!}=c\int_M\frac{\omega_0^n}{n!}-\int_M \chi\wedge\frac{\omega_0^{n-1}}{(n-1)!}\ge0,\]
then there exists $\omega_\varphi=\omega_0+\sqrt{-1}\partial\bar\partial\varphi$ satisfying the equation
\[\mathrm{tr}_{\omega_\varphi}\chi+f\frac{\chi^n}{\omega_\varphi^n}=c\] and the inequality \[c\omega_\varphi^{n-1}-(n-1)\chi\wedge\omega_\varphi^{n-2}>0\] if \[c\omega_0^{n-1}-(n-1)\chi\wedge\omega_0^{n-2}>0.\]
\label{Solvability-assuming-sub-solution}
\end{theorem}

We will use the continuity method to prove Theorem \ref{Solvability-assuming-sub-solution}. The details will be provided in Section 2.

\begin{remark}
Let $\chi=\delta_{ij}$ and $\omega_\varphi=\lambda_i\delta_{ij}$, then the equation \[\mathrm{tr}_{\omega_\varphi}\chi+f\frac{\chi^n}{\omega_\varphi^n}=c\] is equivalent to
\[\sum_{i=1}^{n}\frac{1}{\lambda_i}+\frac{f}{\prod_{i=1}^{n}\lambda_i}=c.\]
\end{remark}

\begin{remark}
Suppose $\sum_{i\not=k}\frac{1}{\lambda_i}\le c$ for all $k=1,2,...,n$ and \[\sum_{i=1}^{n}\frac{1}{\lambda_i}+\frac{f}{\prod_{i=1}^{n}\lambda_i}=c,\] then as long as $f>-\frac{1}{2n}(\frac{1}{c})^{n-1}$, it is easy to see that $\sum_{i\not=k}\frac{1}{\lambda_i}<c$.
\label{Nondegeneracy}
\end{remark}

\begin{remark}
When $n=2$, Theorem \ref{Solvability-assuming-sub-solution} is the Calabi conjecture solved by Yau \cite{Yau}. When $f=0$, Theorem \ref{Solvability-assuming-sub-solution} is a speical case of Song and Weinkove's result \cite{SongWeinkove}. When $f$ is a constant times $\frac{\omega_0^n}{\chi^n}$, Theorem \ref{Solvability-assuming-sub-solution} was proved by Zheng \cite{Zheng}.
\end{remark}

Step 2: Prove the following:
\begin{theorem}
Fix a K\"ahler manifold $M^n$ with K\"ahler metrics $\chi$ and $\omega_0$. Suppose that for all $t>0$, there exist a constant $c_t>0$ and a smooth K\"ahler form $\omega_t\in[(1+t)\omega_0]$ satisfying \[c\omega_t^{n-1}-(n-1)\chi\wedge\omega_t^{n-2}>0\] and \[\mathrm{tr}_{\omega_t}\chi+c_t\frac{\chi^n}{\omega_t^n}=c.\] Then there exist a constant $\epsilon_{1.4}>0$ and a current $\omega_{1.5}\in[\omega_0-\epsilon_{1.4}\chi]$ such that \[c\omega_{1.5}^{n-1}-(n-1)\chi\wedge\omega_{1.5}^{n-2}\ge0\] in the sense of Definition \ref{Definition-current-sub-solution}.
\label{Current-sub-solution}
\end{theorem}
\begin{remark}
In general we can only take the wedge product of $\omega_\varphi$ when $\varphi$ is in $C^2$. Bedford-Taylor \cite{BedfordTaylor} proved that it can also be defined when $\varphi$ is in $L^\infty$. In our case, $\varphi$ is unbounded, so we have to figure out the correct definition of \[c\omega_\varphi^p-p\chi\wedge\omega_\varphi^{p-1}\ge0\] for unbounded $\varphi$ and $p=1, 2, ..., n$. This will be done in Definition \ref{Definition-current-sub-solution}.
\end{remark}
\begin{remark}
When $n=2$, it is same as Theorem 2.12 of \cite{DemaillyPaun}.
\end{remark}

Now let us sketch the proof here. It is analogous to the proof of Theorem 2.12 of Demailly-Paun's paper \cite{DemaillyPaun}. Consider the diagonal $\Delta$ inside the product manifold $M\times M$. Cover it by finitely many open coordinate balls $B_j$. Since $\Delta$ is non-singular, we can assume that on $B_j$, $g_{j,k}$, $k=1,2,...,2n$ are coordinates and $\Delta=\{g_{j,k}=0,1\le k\le n\}$. Assume that $\theta_j$ are smooth functions supported in $B_j$ such that $\sum\theta_j^2=1$ in a neighborhood of $\Delta$. For $\epsilon_{1.6}>0$, define \[\psi_{\epsilon_{1.6}}=\log(\sum_j\theta_j^2\sum_{k=1}^{n}|g_{j,k}|^2+\epsilon_{1.6}^2).\]
Define \[\chi_{M\times M}=\pi_1^*\chi+\pi_2^*\chi,\] and \[\chi_{M\times M,\epsilon_{1.6},\epsilon_{1.7}}=\chi_{M\times M}+\epsilon_{1.7}\sqrt{-1}\partial\bar\partial\psi_{\epsilon_{1.6}}.\] Let \[f_{t,\epsilon_{1.6},\epsilon_{1.7}}=\frac{\chi_{M\times M,\epsilon_{1.6},\epsilon_{1.7}}^{2n}}{\chi_{M\times M}^{2n}}-1+\frac{c_t}{c^n}>\frac{\chi_{M\times M,\epsilon_{1.6},\epsilon_{1.7}}^{2n}}{\chi_{M\times M}^{2n}}-1,\] then by Lemma 2.1 (ii) of \cite{DemaillyPaun}, there exists $\epsilon_{1.7}>0$ such that for $\epsilon_{1.6}$ small enough, \[\frac{\chi_{M\times M,\epsilon_{1.6},\epsilon_{1.7}}^{2n}}{\chi_{M\times M}^{2n}}-1>-\frac{1}{4n}(\frac{1}{(n+1)c})^{2n-1}.\] Now we consider $\omega_{0,M\times M,t}=\pi_1^*\omega_t+\frac{1}{c}\pi_2^*\chi$. By Theorem  \ref{Solvability-assuming-sub-solution}, there exists $\omega_{t,\epsilon_{1.6},\epsilon_{1.7}}\in[\omega_{0,M\times M,t}]$ such that
\[\mathrm{tr}_{\omega_{t,\epsilon_{1.6},\epsilon_{1.7}}}\chi_{M\times M}+f_{t,\epsilon_{1.6},\epsilon_{1.7}}\frac{\chi_{M\times M}^{2n}}{\omega_{t,\epsilon_{1.6},\epsilon_{1.7}}^{2n}}=(n+1)c.\]

Now define $\omega_{1,t,\epsilon_{1.6},\epsilon_{1.7}}$ by \[\omega_{1,t,\epsilon_{1.6},\epsilon_{1.7}}=\frac{c^{n-1}}{\int_M n\chi^n}(\pi_1)_*(\omega_{t,\epsilon_{1.6},\epsilon_{1.7}}^n\wedge\pi_2^*\chi).\] Fix $\epsilon_{1.7}$ and let $t$ and $\epsilon_{1.6}$ converge to 0. For small enough $\epsilon_{1.4}$, let $\omega_{1.5}$ be the weak limit of $\omega_{1,t,\epsilon_{1.6},\epsilon_{1.7}}-\epsilon_{1.4}\chi$. Then we shall expect \[c\omega_{1.5}^{n-1}-(n-1)\chi\wedge\omega_{1.5}^{n-2}\ge0\] in the sense of Definition \ref{Definition-current-sub-solution}. The details will be provided in Section 3.

Step 3:
Consider the set $I$ of $t\ge 0$ such that there exist a constant $c_t\ge 0$ and a smooth K\"ahler form $\omega_t\in[(1+t)\omega_0]$ satisfying \[(c\omega_t-(n-1)\chi)\wedge\omega_t^{n-2}>0\] and \[\mathrm{tr}_{\omega_t}\chi+c_t\frac{\chi^n}{\omega_t^n}=c.\] By Theorem \ref{Solvability-assuming-sub-solution}, it suffices to show that $0\in I$. When $t$ is large enough, the condition of Theorem \ref{Solvability-assuming-sub-solution} is satisfied. So $t\in I$. It is easy to see that if $t\in I$, then for nearby $t$, the condition of Theorem \ref{Solvability-assuming-sub-solution} is also satisfied. So $I$ is open. Still by Theorem \ref{Solvability-assuming-sub-solution}, as long as $t\in I$, then for all $t'\ge t$, $t'\in I$. Thus, in order to prove the closedness of $I$, it suffices to show that if $t\in I$ for all $t>t_0$, then $t_0\in I$. After replacing $(1+t_0)\omega_0$ by $\omega_0$, we can without loss of generality assume that $t_0=0$. In particular we can apply Theorem \ref{Current-sub-solution} to get $\omega_{1.5}$.

Let $\nu(x)$ be the Lelong number of $\omega_{1.5}$ at $x$. For $\epsilon_{1.8}>0$ to be determined, let $Y$ be the set \[Y=\{x:\nu(x)\ge\epsilon_{1.8}\}.\] By the result of Siu \cite{Siu}, $Y$ is a subvariety with dimension $p<n$. Assume that $Y$ is smooth, then by induction hypothesis, we can apply Theorem \ref{Main-induction} to $Y$ to obtain a smooth function $\varphi_{1.9}$ on $Y$ such that $\omega_{1.9}=\omega_0|_Y+\sqrt{-1}\partial\bar\partial\varphi_{1.9}$ satisfies \[(c-(n-p)\epsilon_{1.1})\omega_{1.9}^p-p\chi|_{Y}\wedge\omega_{1.9}^{p-1}\ge0\] on $Y$. Then for large enough $C_{1.10}$, \[\omega_{1.11}=\omega_0+\sqrt{-1}\partial\bar\partial (\mathrm{Proj}_Y^*\varphi_{1.9}+C_{1.10}d_\chi(.,Y)^2)\] satisfies \[(c-\frac{n-p}{2}\epsilon_{1.1})\omega_{1.11}^{n-1}-(n-1)\chi\wedge\omega_{1.11}^{n-2}>0\] on a tubular neighborhood of $Y$, where $\mathrm{Proj}_Y$ means the projection to $Y$. By a generalization of the result of B\l{}ocki and Ko\l{}odziej \cite{BlockiKolodziej}, we can glue the smoothing of $\omega_{1.5}$ outside $Y$ and $\omega_{1.11}$ near $Y$ into $\omega_{1.12}=\omega_0+\sqrt{-1}\partial\bar\partial\varphi_{1.12}$ satisfying \[c\omega_{1.12}^{n-1}-(n-1)\chi\wedge\omega_{1.12}^{n-2}>0\] on $M$. Then we are done by Theorem \ref{Solvability-assuming-sub-solution}. In general, $Y$ is singular and we need to use Hironaka's desingularization theorem to resolve it. The details will be provided in Section 4.

\section*{Acknowledgement}
The author wishes to thank Xiuxiong Chen for suggesting him this problem and providing valuable comments. The author is also grateful to Jingrui Cheng for pointing out a gap in the first version of this paper, Helmut Hofer for a discussion about symplectic geometry as well as Simone Calamai and Long Li for minor suggestions. This material is based upon
work supported by the National Science Foundation under Grant No. 1638352, as well as support from the S. S. Chern Foundation for Mathematics Research Fund.

\section{The analysis part}
In this section, we use the continuity method twice to prove Theorem \ref{Solvability-assuming-sub-solution}. First of all, for $t\in[0,1]$, define $\chi_t$ by \[\chi_t=t\chi+(1-t)\frac{c}{n}\omega_0\] and define $f_t\ge0$ as the constant such that \[\int_M f_t\frac{\chi_t^n}{n!}=c\int_M\frac{\omega_0^n}{n!}-\int_M \chi_t\wedge\frac{\omega_0^{n-1}}{(n-1)!}=t(c\int_M\frac{\omega_0^n}{n!}-\int_M \chi\wedge\frac{\omega_0^{n-1}}{(n-1)!})\ge0.\] Now we consider the set $I$ consisting of all $t\in[0,1]$ such that there exists $\omega_t=\omega_0+\sqrt{-1}\partial\bar\partial\varphi_t>0$ for smooth $\varphi_t$ satisfying \[\mathrm{tr}_{\omega_t}\chi_t+f_t\frac{\chi_t^n}{\omega_t^n}=c\] and \[c\omega_t^{n-1}-(n-1)\chi_t\wedge\omega_t^{n-2}>0.\] Then it is easy to see that $0\in I$. Remark that the equation is the same as \[c\frac{\omega_t^n}{n!}-\chi_t\wedge\frac{\omega_t^{n-1}}{(n-1)!}=f_t\frac{\chi_t^n}{n!}.\]
The linearization is \[\frac{1}{(n-1)!}(c\omega_t^{n-1}-(n-1)\chi_t\wedge\omega_t^{n-2})\wedge \sqrt{-1}\partial\bar\partial\frac{\partial\varphi_t}{\partial t}=\frac{\partial}{\partial t}(f_t\frac{\chi_t^n}{n!})+\frac{\partial\chi_t}{\partial t}\wedge\frac{\omega_t^{n-1}}{(n-1)!}.\] Assume that $t\in I$, then the left hand side is a second order elliptic equation on $\frac{\partial\varphi_t}{\partial t}$. On the other hands, the integrability condition implies that the integral of the right hand side is 0. By standard elliptic theory and the implicit function theorem, $I$ is open when we replace the smoothness assumption of $\varphi$ by $C^{100,\alpha}$. However, standard elliptic regularity theory implies that any $C^{100,\alpha}$ solution is automatically smooth. So $I$ is in fact open.

Assume that we are able to show the closedness of $I$, then we have proved Theorem \ref{Solvability-assuming-sub-solution} for $f$ replaced by $f_1$. We can use another continuity path by fixing $\chi$ and $\omega_0$ but choosing $\hat f_t=tf_1+(1-t)f$. However, it is the same as before except that $\hat f_t>-\frac{1}{2n}(\frac{1}{c})^{n-1}$ is a function instead of a constant. Thus, we only need to prove the \textit{a priori} estimate of $\omega_t$ by assuming that $f_t>-\frac{1}{2n}(\frac{1}{c})^{n-1}$ is a function. We start from the following proposition which is analogous to Lemma 3.1 of Song-Weinkove's paper \cite{SongWeinkove}:
\begin{proposition}
Assume that $t\in I$ and $\omega_t=\omega_0+\sqrt{-1}\partial\bar\partial\varphi_t$ is the corresponding solution, then there exist constants $C_{2.1}$ and $C_{2.2}$ depending only on $c$, $\omega_0$, the $C^\infty$-norm of $\chi_t$ with respect to $\omega_0$, the $C^2$-norm of $||f_t||$ with respect to $\omega_0$ such that
\[\mathrm{tr}_{\chi_t}\omega_t\le C_{2.2} e^{C_{2.1}(\varphi_t-\inf\varphi_t)}.\]
\label{C2-estimate}
\end{proposition}
\begin{proof}
In local coordinates, $\chi_t=\sqrt{-1}\chi_{i\bar j}dz^i\wedge dz^{\bar j}$ and $\omega_t=\sqrt{-1} g_{i\bar j}dz^i\wedge dz^{\bar j}$. Fix any point $x$, choose a $\chi_t$-normal coordinate such that $\chi_{i\bar j}=\delta_{i\bar j}$, $\chi_{i\bar j,k}=\chi_{i\bar j,\bar k}=0$ and $g_{i\bar j}=\lambda_i\delta_{i\bar j}$ at $x$, where the derivatives are all ordinary derivatives. Then the equation is
\[\sum_{i,j}g^{i\bar j}\chi_{i\bar j}+f_t\frac{\det\chi_{\alpha\bar\beta}}{\det g_{\alpha\bar\beta}}=c.\]
Define an operator $\tilde\Delta$ by
\[\tilde\Delta u=\sum_{i,l}(\sum_{j,k}g^{i\bar j}g^{k\bar l}\chi_{k\bar j}+f_t\frac{g^{i\bar l}\det\chi_{\alpha\bar\beta}}{\det g_{\alpha\bar\beta}})u_{,i\bar l},\]
then it is easy to see that $\tilde\Delta$ is independent of the choice of local coordinates.

At $x$, \[\tilde\Delta u=\sum_i(\frac{1}{\lambda_i^2}+f_t\frac{1}{\lambda_i}\frac{1}{\prod_{\alpha=1}^{n}\lambda_\alpha})u_{,i\bar i}.\] Since $\frac{1}{\lambda_\alpha}<c$ and $f_t>-\frac{1}{2n}(\frac{1}{c})^{n-1}$, it is easy to see that \[\frac{1}{\lambda_i^2}+f_t\frac{1}{\lambda_i}\frac{1}{\prod_{\alpha=1}^{n}\lambda_\alpha}>0\] for all $i$. So $\tilde\Delta$ is a second order elliptic operator.

Now we compute $\tilde\Delta(\log\mathrm{tr}_{\chi_t}\omega_t)=\tilde\Delta(\log(\sum_{i,j}\chi^{i\bar j}g_{i\bar j}))$. It equals to
\[\sum_k(\frac{\sum_i (g_{i\bar i,k\bar k}+(\chi^{i\bar i})_{,k\bar k}\lambda_i)}{\sum_i\lambda_i}-\frac{|\sum_ig_{i\bar i,k}|^2}{(\sum_i\lambda_i)^2})(\frac{1}{\lambda_k^2}+f_t\frac{1}{\lambda_k}\frac{1}{\prod_{\alpha=1}^{n}\lambda_\alpha}))\] at $x$.

Now we differentiate the equation \[\sum_{i,j}g^{i\bar j}\chi_{i\bar j}+f_t\frac{\det\chi_{\alpha\bar\beta}}{\det g_{\alpha\bar\beta}}=c,\] then \[\begin{split}\sum_{i,j}g^{i\bar j}\chi_{i\bar j,k}-\sum_{i,j,a,b}g^{i\bar b}g_{a\bar b,k}g^{a\bar j}\chi_{i\bar j}+\frac{\det\chi_{\alpha\bar\beta}}{\det g_{\alpha\bar\beta}}(f_{t,k}\\+f_t\sum_{i,j}(\chi^{i\bar j}\chi_{i\bar j,k}-g^{i\bar j}g_{i\bar j,k}))=0.\end{split}\]
So \[\begin{split}\frac{1}{\sum_i\lambda_i}\sum_{k}(\sum_i\frac{1}{\lambda_i}\chi_{i\bar i,k\bar k}-\sum_{i}\frac{1}{\lambda_i^2}g_{i\bar i,k\bar k}+\sum_{i,j}\frac{1}{\lambda_i^2}\frac{1}{\lambda_j}(|g_{i\bar j,k}|^2+|g_{i\bar j,\bar k}|^2)\\+\frac{1}{\prod_{\alpha}\lambda_\alpha}(f_t(|\sum_{i}(\frac{1}{\lambda_i}g_{i\bar i,k})|^2+\sum_i\chi_{i\bar i,k\bar k}+\sum_{i,j}\frac{1}{\lambda_i\lambda_j}|g_{i\bar j,k}|^2-\sum_{i}\frac{1}{\lambda_i}g_{i\bar i,k\bar k})\\+f_{t,k\bar k}-f_{t,\bar k}\sum_{i}(\frac{1}{\lambda_i}g_{i\bar i,k})-f_{t,k}\sum_{i}(\frac{1}{\lambda_i}g_{i\bar i,\bar k})))=0\end{split}\] at $x$.

By K\"ahler condition, $g_{i\bar i,k\bar k}=g_{k\bar k,i\bar i}$, $g_{i\bar j,k}=g_{k\bar j,i}$ and $g_{i\bar j,\bar k}=g_{i\bar k,\bar j}$. Using the bounds that $|\chi_{i\bar i,k\bar k}|+|(\chi^{i\bar i})_{,k\bar k}|+|f_{t,k}|+|f_{t,\bar k}|+|f_{t,k\bar k}|+|f_t|+\frac{1}{\lambda_i}<C_{2.3}$ for all $i$, $k$, it is easy to see that by taking the sum of the previous two equations,
\[\begin{split}\tilde\Delta(\log\mathrm{tr}_\chi\omega_t)\ge-C_{2.4}-\sum_k(\frac{|\sum_ig_{i\bar i,k}|^2}{(\sum_i\lambda_i)^2})(\frac{1}{\lambda_k^2}+f_t\frac{1}{\lambda_k}\frac{1}{\prod_{\alpha=1}^{n}\lambda_\alpha})\\+\frac{1}{\sum_i\lambda_i}\sum_{k}(\sum_{i,j}\frac{1}{\lambda_i^2}\frac{1}{\lambda_j}(|g_{i\bar j,k}|^2+|g_{i\bar j,\bar k}|^2)+\frac{1}{\prod_{\alpha}\lambda_\alpha}(f_t(|\sum_{i}(\frac{1}{\lambda_i}g_{i\bar i,k})|^2\\+\sum_{i,j}\frac{1}{\lambda_i\lambda_j}|g_{i\bar j,k}|^2)-f_{t,\bar k}\sum_{i}(\frac{1}{\lambda_i}g_{i\bar i,k})-f_{t,k}\sum_{i}(\frac{1}{\lambda_i}g_{i\bar i,\bar k}))).\end{split}\]
Remark that \[\begin{split}|\frac{1}{\prod_{\alpha}\lambda_\alpha}f_{t,\bar k}\sum_{i}(\frac{1}{\lambda_i}g_{i\bar i,k})|=|\frac{1}{\prod_{\alpha}\lambda_\alpha}f_{t,k}\sum_{i}(\frac{1}{\lambda_i}g_{i\bar i,\bar k})|\\=|f_{t,\bar k}||\sum_{i}\frac{1}{\prod_{\alpha}\lambda_\alpha}(\frac{1}{\lambda_i}g_{i\bar i,k})|\le C_{2.5}\sum_{i}|\frac{1}{\lambda_i^2}g_{i\bar i,k}|\\\le \frac{1}{4}\sum_{i}\frac{1}{\lambda_i^3}|g_{i\bar i,k}|^2+C_{2.6}\sum_{i}\frac{1}{\lambda_i}\le\frac{1}{4}\sum_{i}\frac{1}{\lambda_i^3}|g_{i\bar i,\bar k}|^2+C_{2.7}\end{split}\]
and \[\begin{split}-\frac{f_t}{\prod_{\alpha}\lambda_\alpha}|\sum_{i}(\frac{1}{\lambda_i}g_{i\bar i,k})|^2\le -\frac{nf_t}{\prod_{\alpha}\lambda_\alpha}\sum_{i}\frac{1}{\lambda_i^2}|g_{i\bar i,k}|^2\le \frac{1}{2}\sum_{i}\frac{1}{\lambda_i^3}|g_{i\bar i,\bar k}|^2.\end{split}\]
So \[\begin{split}\tilde\Delta(\log\mathrm{tr}_{\chi_t}\omega_t)\ge-C_{2.8}-\sum_k(\frac{|\sum_ig_{i\bar i,k}|^2}{(\sum_i\lambda_i)^2})(\frac{1}{\lambda_k^2}+f_t\frac{1}{\lambda_k}\frac{1}{\prod_{\alpha=1}^{n}\lambda_\alpha})\\+\frac{1}{\sum_i\lambda_i}\sum_{k}(\sum_{i,j}\frac{1}{\lambda_i^2}\frac{1}{\lambda_j}|g_{i\bar j,k}|^2+\frac{f_t}{\prod_{\alpha}\lambda_\alpha}\sum_{i,j}\frac{1}{\lambda_i\lambda_j}|g_{i\bar j,k}|^2).\end{split}\]
We have used \[\sum_{i,j}\frac{1}{\lambda_i^2}\frac{1}{\lambda_j}|g_{i\bar j,\bar k}|^2\ge\sum_{i}\frac{1}{\lambda_i^3}|g_{i\bar i,\bar k}|^2\] here.

By Cauchy-Schwarz inequality and the fact that $g_{i\bar i,k}=g_{k\bar i,i}$,
\[\begin{split}&\sum_k(|\sum_ig_{i\bar i,k}|^2)(\frac{1}{\lambda_k^2}+f_t\frac{1}{\lambda_k}\frac{1}{\prod_{\alpha=1}^{n}\lambda_\alpha})\\
&\le\sum_{i,j,k}|g_{i\bar i,k}||g_{j\bar j,k}|(\frac{1}{\lambda_k^2}+f_t\frac{1}{\lambda_k}\frac{1}{\prod_{\alpha=1}^{n}\lambda_\alpha})\\
&\le\sum_{i,j}\sqrt{\sum_k|g_{i\bar i,k}|^2(\frac{1}{\lambda_k^2}+f_t\frac{1}{\lambda_k}\frac{1}{\prod_{\alpha=1}^{n}\lambda_\alpha})}\sqrt{\sum_k|g_{j\bar j,k}|^2(\frac{1}{\lambda_k^2}+f_t\frac{1}{\lambda_k}\frac{1}{\prod_{\alpha=1}^{n}\lambda_\alpha})}\\
&=(\sum_{i}\sqrt{\sum_k|g_{i\bar i,k}|^2(\frac{1}{\lambda_k^2}+f_t\frac{1}{\lambda_k}\frac{1}{\prod_{\alpha=1}^{n}\lambda_\alpha})})^2\\
&\le(\sum_{i}\lambda_i)\sum_{i}\sum_k\frac{|g_{i\bar i,k}|^2}{\lambda_i}(\frac{1}{\lambda_k^2}+f_t\frac{1}{\lambda_k}\frac{1}{\prod_{\alpha=1}^{n}\lambda_\alpha})\\
&=(\sum_{i}\lambda_i)\sum_{i,k}\frac{|g_{i\bar k,k}|^2}{\lambda_k}(\frac{1}{\lambda_i^2}+f_t\frac{1}{\lambda_i}\frac{1}{\prod_{\alpha=1}^{n}\lambda_\alpha})\\
&\le(\sum_{i}\lambda_i)\sum_{i,j,k}\frac{|g_{i\bar j,k}|^2}{\lambda_j}(\frac{1}{\lambda_i^2}+f_t\frac{1}{\lambda_i}\frac{1}{\prod_{\alpha=1}^{n}\lambda_\alpha}),
\end{split}\]
so $\tilde\Delta(\log\mathrm{tr}_{\chi_t}\omega_t)\ge-C_{2.9}$ at $x$. However, since $x$ is arbitrary and $\tilde\Delta$ is independent of the local coordinates, we see that $\tilde\Delta(\log\mathrm{tr}_{\chi_t}\omega_t)\ge-C_{2.9}$ on $M$.

Choose $\epsilon_{2.10}<\frac{c}{2n}$ as a small constant such that \[c\omega_0^{n-1}-(n-1)\chi\wedge\omega_0^{n-2}>2\epsilon_{2.10}\omega_0^{n-1},\]
then
\[c\omega_0^{n-1}-(n-1)\chi_t\wedge\omega_0^{n-2}>2\epsilon_{2.10}\omega_0^{n-1}\] by the definition of $\chi_t$. Choose $C_{2.1}$ as $\frac{2C_{2.9}}{\epsilon_{2.10}}$, then at the maximal point of $\log\mathrm{tr}_{\chi_t}\omega_t-C_{2.1}\varphi_t$, \[-\tilde\Delta\varphi_t=-\sum_{i,l}(\sum_{j,k}g^{i\bar j}g^{k\bar l}\chi_{k\bar j}+f_t\frac{g^{i\bar l}\det\chi_{\alpha\bar\beta}}{\det g_{\alpha\bar\beta}})(g_{i\bar l}-g^0_{i\bar l})<\frac{\epsilon_{2.10}}{2}.\]
If \[c-\sum_{i,l}\sum_{j,k}g^{i\bar j}g^{k\bar l}\chi_{k\bar j}(2g_{i\bar l}-g^0_{i\bar l})<\epsilon_{2.10},\]
by the proof of Lemma 3.1 of \cite{SongWeinkove}, $\mathrm{tr}_{\chi_t}\omega_t\le C_{2.11}$. If \[c-\sum_{i,l}\sum_{j,k}g^{i\bar j}g^{k\bar l}\chi_{k\bar j}(2g_{i\bar l}-g^0_{i\bar l})\ge\epsilon_{2.10},\] then \[-\sum_{i,l}f_t\frac{g^{i\bar l}\det\chi_{\alpha\bar\beta}}{\det g_{\alpha\bar\beta}}(g_{i\bar l}-g^0_{i\bar l})<-\frac{\epsilon_{2.10}}{2}+c-\sum_{i,j}g^{i\bar j}\chi_{i\bar j}=-\frac{\epsilon_{2.10}}{2}+f_t\frac{\det\chi_{\alpha\bar\beta}}{\det g_{\alpha\bar\beta}},\]
so $\prod_{i}\lambda_i=\frac{\det g_{\alpha\bar\beta}}{\det\chi_{\alpha\bar\beta}}<C_{2.12}$. Using the fact that $\lambda_i>\frac{1}{c}$, $\mathrm{tr}_{\chi_t}\omega_t=\sum_i\lambda_i\le C_{2.13}$ is also true. This completes the proof of the proposition.
\end{proof}

By adding a constant if necessary, we can without loss of generality assume that $\sup_M\varphi_t=0$. Then we have the following $C^0$ estimate:
\begin{proposition}
\[||\varphi_t||_{C^0}\le C_{2.14}.\] Moreover, $C_{2.15}^{-1}\chi_t\le\omega_t\le C_{2.15}\chi_t$.
\label{C0-estimate}
\end{proposition}
\begin{proof}
Lemma 3.3 and Lemma 3.4 and Proposition 3.5 of \cite{SongWeinkove} only used the inequality in Proposition \ref{C2-estimate}. So they are still true in our case.
\end{proof}

\begin{proposition}
$I$ is closed.
\label{Closedness}
\end{proposition}
\begin{proof}
First of all, we want to check the uniform ellipticity and concavity for the Evans-Krylov estimate. The equation is \[-g^{i\bar j}\chi_{i\bar j}-f_t\frac{\det\chi_{\alpha\bar\beta}}{\det g_{\alpha\bar\beta}}=-c.\]
View it as a function in terms of $g_{i\bar j}$, $\chi_{i\bar j}$ and $f_t$, then the partial derivative in the $g_{a\bar b}$ direction is
\[g^{i\bar b}g^{a\bar j}\chi_{i\bar j}+f_t\frac{\det\chi_{\alpha\bar\beta}}{\det g_{\alpha\bar\beta}}g^{a\bar b}.\]
At $x$, it equals to \[(\frac{1}{\lambda_a^2}+\frac{1}{\lambda_a}\frac{f_t}{\prod_i\lambda_i})\delta_{a\bar b}.\]
It has uniform upper bound and lower bound.

The second order derivative in $g_{a\bar b}$ and $g_{c\bar d}$ direction is
\[-g^{i\bar d}g^{c\bar b}g^{a\bar j}\chi_{i\bar j}-g^{i\bar b}g^{a\bar d}g^{c\bar j}\chi_{i\bar j}-f_t\frac{\det\chi_{\alpha\bar\beta}}{\det g_{\alpha\bar\beta}}g^{a\bar b}g^{c\bar d}-f_t\frac{\det\chi_{\alpha\bar\beta}}{\det g_{\alpha\bar\beta}}g^{a\bar d}g^{c\bar b}.\]
At $x$, when taking the product with $w_{a\bar b}\overline{w_{c\bar d}}$ and summing $a,b,c,d$ for any matrix $w_{i\bar j}$, we get \[-\sum_{a,b}\frac{1}{\lambda_a^2\lambda_b}|w_{a\bar b}|^2-\sum_{a,b}\frac{1}{\lambda_b^2\lambda_a}|w_{a\bar b}|^2-\frac{f_t}{\prod_i\lambda_i}(\sum_a\frac{w_{a\bar a}}{\lambda_a})^2-\frac{f_t}{\prod_i\lambda_i}\sum_{a,b}\frac{1}{\lambda_a\lambda_b}|w_{a\bar b}|^2.\]
It is easy to see that it is non-positive.

Thus, if we replace the complex second derivatives by real second derivatives, the uniform ellipticity and concavity for the Evans-Krylov estimate \cite{Evans1, Evans2, Krylov, Trudinger} are satisfied. By checking Evans-Krylov's estimate carefully, it is easy to see that in our complex case, the estimate \[[(\varphi_t)_{i\bar j}]_{C^\alpha}\le C_{2.16}\] is still true.

By standard elliptic estimate, $||\varphi_t||_{C^{101,\alpha}}$ is bounded. By Arzela-Ascoli theorem, if $t_i\rightarrow t_\infty$ and $t_i\in I$, then a subsequence of $\varphi_t$ converges to $\varphi_{t_\infty}$ in $C^{100,\alpha}$-norm. By Remark \ref{Nondegeneracy}, \[c\omega_{t_\infty}^{n-1}-(n-1)\chi_t\wedge\omega_{t_\infty}^{n-2}>0.\]
So by standard elliptic regularity, $\varphi_{t_\infty}$ is smooth. In other words, $t_\infty\in I$.
\end{proof}

\section{Concentration of mass and its application}
In this section, we prove Theorem \ref{Current-sub-solution}. However, before that, we need to figure out the correct definition of \[c\omega^p-p\chi\wedge\omega^{p-1}\ge0\] when $\omega$ is only a current.

Recall the following definition of the smoothing:

\begin{definition}
Fix a smooth non-negative function $\rho$ supported in [0,1] such that
\[\int_0^1\rho(t)t^{2n-1}\mathrm{Vol}(\partial B_1(0))dt=1\] and $\rho$ is a positive constant near 0.
For any $\delta>0$, the smoothing $\varphi_\delta$ is defined by
\[\varphi_\delta(x)=\int_{\mathbb{C}^n} \varphi(x-y)\delta^{-2n}\rho(|\frac{y}{\delta}|)d\mathrm{Vol}_y.\]
We can define the smoothing of a current using similar formula. It is easy to see that the smoothing commutes with derivatives. So $(\sqrt{-1}\partial\bar\partial\varphi)_\delta=\sqrt{-1}\partial\bar\partial(\varphi_\delta)$.
\label{Definition-smoothing}
\end{definition}

Recall that $\sqrt{-1}\partial\bar\partial\varphi\ge 0$ if and only if $\sqrt{-1}\partial\bar\partial\varphi_\delta\ge0$ for all $\delta>0$. As an analogy, we can define \[c\omega^p-p\chi\wedge\omega^{p-1}\ge0\] for a closed positive (1,1) current $\omega$ using smoothing. Remark that any closed positive (1,1) current can be written as $\sqrt{-1}\partial\bar\partial$ acting on a real function locally.
\begin{definition}
Suppose that $\chi$ is a K\"ahler form with constant coefficients on an open set $O\subset\mathbb{C}^n$. Then we say that \[c(\sqrt{-1}\partial\bar\partial\varphi)^{p}-p\chi\wedge(\sqrt{-1}\partial\bar\partial\varphi)^{p-1}\ge0\] on $O$ if for any $\delta>0$, the smoothing $\varphi_\delta$ satisfies  \[c(\sqrt{-1}\partial\bar\partial\varphi_\delta)^p-p\chi\wedge(\sqrt{-1}\partial\bar\partial\varphi_\delta)^{p-1}\ge0\] on the set $O_\delta=\{x:B_\delta(x)\subset O\}$.
\end{definition}

We can also define it without the constant coefficients assumption.
\begin{definition}
We say that \[c\omega^p-p\chi\wedge\omega^{p-1}\ge0\] if on any coordinate chart, for any open subset $O$, as long as $\chi\ge\chi_0$ on $O$ for a K\"ahler form $\chi_0$ with constant coefficients, then \[c\omega^p-p\chi_0\wedge\omega^{p-1}\ge0.\]
\label{Definition-current-sub-solution}
\end{definition}

\begin{remark}
\[c(\sqrt{-1}\partial\bar\partial\varphi)^{p}-p\chi_0\wedge(\sqrt{-1}\partial\bar\partial\varphi)^{p-1}\ge0\] is a convex property for $\varphi$. So if $\omega$ is smooth, then \[c\omega^p-p\chi\wedge\omega^{p-1}\ge0\] on $O$ pointwise if and only if it is true on $O$ in the sense of Definition \ref{Definition-current-sub-solution}.
\end{remark}

For simplicity, for any positive definite $n\times n$ matrix $A$, we define $P_I(A)$ by
\[P_I(A)=\max_{k}(\sum_{j\not=k}\frac{1}{\lambda_j})=\max_{V^{n-1}\subset \mathbb{C}^n}(\mathrm{tr}(A|_{V})^{-1}),\] where $\lambda_j$ are the eigenvalues of $A$. Then \[c\omega^{n-1}-(n-1)\chi\wedge\omega^{n-2}\ge0\] is equivalent to $P_\chi(\omega)\le c$.

Now we need a lemma:
\begin{lemma}
\[P_I(A-CB^{-1}\bar C^{T})+\mathrm{tr}(B^{-1})\le P_I(\begin{bmatrix}
    A       & C \\
    \bar C^{T} & B
\end{bmatrix})\]
\label{Subadditive-submatrix}
\end{lemma}
\begin{proof}
By restricting on the codimension 1 subspaces, it suffices to prove that
\[\mathrm{tr}((A-CB^{-1}\bar C^{T})^{-1})+\mathrm{tr}(B^{-1})\le \mathrm{tr}(\begin{bmatrix}
    A       & C \\
    \bar C^{T} & B
\end{bmatrix}^{-1}).\]

It is easy to see that
\[\begin{bmatrix}
    I   & -CB^{-1} \\
    O & I
\end{bmatrix}\begin{bmatrix}
    A               & C \\
    \bar C^{T} & B
\end{bmatrix}\begin{bmatrix}
    I                            & O \\
    -B^{-1}\bar C^{T} & I
\end{bmatrix}=\begin{bmatrix}
    A-CB^{-1}\bar C^{T} & O \\
    O                               & B
\end{bmatrix}.\]
So
\[\begin{bmatrix}
    I                            & O \\
    -B^{-1}\bar C^{T} & I
\end{bmatrix}\begin{bmatrix}
    A-CB^{-1}\bar C^{T} & O \\
    O                               & B
\end{bmatrix}^{-1}\begin{bmatrix}
    I   & -CB^{-1} \\
    O & I
\end{bmatrix}=\begin{bmatrix}
    A               & C \\
    \bar C^{T} & B
\end{bmatrix}^{-1}.\]
After taking traces, the left hand side equals to
\[\mathrm{tr}((A-CB^{-1}\bar C^{T})^{-1})+\mathrm{tr}(B^{-1})+\mathrm{tr}(B^{-1}\bar C^{T}(A-CB^{-1}\bar C^{T})^{-1}CB^{-1}).\]
\end{proof}

Now we start the proof of Theorem \ref{Current-sub-solution}. By assumption, for any $t>0$, there exist $c_t>0$ and $\omega_t\in[(1+t)\omega_0]$ satisfying \[c\omega_t^{n-1}-(n-1)\chi\wedge\omega_t^{n-2}>0\] and \[\mathrm{tr}_{\omega_t}\chi+c_t\frac{\chi^n}{\omega_t^n}=c.\] Consider $\omega_{0,M\times M,t}=\pi_1^*\omega_t+\frac{1}{c}\pi_2^*\chi$ and $\chi_{M\times M}=\pi_1^*\chi+\pi_2^*\chi$. At each point, diagonalizing them so that $\chi_{i\bar j}=\delta_{i\bar j}$ and $(\omega_t)_{i\bar j}=\lambda_i\delta_{i\bar j}$. Then the eigenvalues on the product manifold are $\lambda_1, ... \lambda_n, \frac{1}{c}, ..., \frac{1}{c}$. Their inverses are $\frac{1}{\lambda_1}, ..., \frac{1}{\lambda_j}, c, ..., c$. So the sum of them is at most $(n+1)c$ because $c_t>0$. In particular, the sum of (2n-1) distinct elements among them is also at most $(n+1)c$. Define $f_{t,\epsilon_{1.6},\epsilon_{1.7}}$ as in Section 1, then there exists $\epsilon_{1.7}>0$ such that for $\epsilon_{1.6}$ small enough, $f_{t,\epsilon_{1.6},\epsilon_{1.7}}>-\frac{1}{4n}(\frac{1}{(n+1)c})^{2n-1}$. So we can apply Theorem \ref{Solvability-assuming-sub-solution} to get $\omega_{t,\epsilon_{1.6},\epsilon_{1.7}}\in[\omega_{0,M\times M,t}]$ such that $P_{\chi_{M\times M}}(\omega_{t,\epsilon_{1.6},\epsilon_{1.7}})<(n+1)c$ and \[\mathrm{tr}_{\omega_{t,\epsilon_{1.6},\epsilon_{1.7}}}\chi_{M\times M}+f_{t,\epsilon_{1.6},\epsilon_{1.7}}\frac{\chi_{M\times M}^{2n}}{\omega_{t,\epsilon_{1.6},\epsilon_{1.7}}^{2n}}=(n+1)c.\]

For each point $(x_1,x_2)$, we assume that $z_1^{(1)}, ... , z_n^{(1)}$ are the local coordinates on $M\times \{x_2\}$, and $z_1^{(2)}, ..., z_n^{(2)}$ are the local coordinates on $\{x_1\}\times M$. Then we can express $\omega_{t,\epsilon_{1.6},\epsilon_{1.7}}$ as \[\omega_{t,\epsilon_{1.6},\epsilon_{1.7}}=\omega^{(1)}_{t,\epsilon_{1.6},\epsilon_{1.7}}+\omega^{(2)}_{t,\epsilon_{1.6},\epsilon_{1.7}}+\omega^{(1,2)}_{t,\epsilon_{1.6},\epsilon_{1.7}}+\omega^{(2,1)}_{t,\epsilon_{1.6},\epsilon_{1.7}},\]
where \[\omega^{(1)}_{t,\epsilon_{1.6},\epsilon_{1.7}}=\sum_{i,j=1}^{n}\sqrt{-1}\omega^{(1)}_{t,\epsilon_{1.6},\epsilon_{1.7},i\bar j}dz_i^{(1)}\wedge d\bar z_j^{(1)},\]
\[\omega^{(2)}_{t,\epsilon_{1.6},\epsilon_{1.7}}=\sum_{i,j=1}^{n}\sqrt{-1}\omega^{(2)}_{t,\epsilon_{1.6},\epsilon_{1.7},i\bar j}dz_i^{(2)}\wedge d\bar z_j^{(2)},\]
\[\omega^{(1,2)}_{t,\epsilon_{1.6},\epsilon_{1.7}}=\sum_{i,j=1}^{n}\sqrt{-1}\omega^{(1,2)}_{t,\epsilon_{1.6},\epsilon_{1.7},i\bar j}dz_i^{(1)}\wedge d\bar z_j^{(2)},\]
\[\omega^{(2,1)}_{t,\epsilon_{1.6},\epsilon_{1.7}}=\overline{\omega^{(1,2)}_{t,\epsilon_{1.6},\epsilon_{1.7}}}.\]
After changing the definition of $z^{(2)}_i$ if necessary, we can assume that \[\pi_2^*\chi=\sqrt{-1}\sum_{i=1}^n dz^{(2)}_i\wedge d\bar z^{(2)}_i\] and \[\omega^{(2)}_{t,\epsilon_{1.6},\epsilon_{1.7}}=\sqrt{-1}\sum_{i=1}^n \lambda_i dz^{(2)}_i\wedge d\bar z^{(2)}_i\] at $(x_1,x_2)$.

Now consider $\omega_{1,t,\epsilon_{1.6},\epsilon_{1.7}}$ defined as \[\begin{split}
&\omega_{1,t,\epsilon_{1.6},\epsilon_{1.7}}\\
&=\frac{c^{n-1}}{\int_M n\chi^n}(\pi_1)_*(\omega_{t,\epsilon_{1.6},\epsilon_{1.7}}^n\wedge\pi_2^*\chi)\\
&=\frac{c^{n-1}}{\int_M n\chi^n}(\pi_1)_*(n\omega_{t,\epsilon_{1.6},\epsilon_{1.7}}^{(1)}\wedge(\omega_{t,\epsilon_{1.6},\epsilon_{1.7}}^{(2)})^{n-1}\wedge\pi_2^*\chi)\\
&+\frac{c^{n-1}}{\int_M n\chi^n}(\pi_1)_*(n(n-1)\omega_{t,\epsilon_{1.6},\epsilon_{1.7}}^{(1,2)}\wedge\omega_{t,\epsilon_{1.6},\epsilon_{1.7}}^{(2,1)}\wedge(\omega_{t,\epsilon_{1.6},\epsilon_{1.7}}^{(2)})^{n-2}\wedge\pi_2^*\chi).\end{split}\]
At $(x_1,x_2)$,
 \[\begin{split}
&n\omega_{t,\epsilon_{1.6},\epsilon_{1.7}}^{(1)}\wedge(\omega_{t,\epsilon_{1.6},\epsilon_{1.7}}^{(2)})^{n-1}\wedge\pi_2^*\chi\\
&=\sum_{i,j=1}^{n}\sqrt{-1}\omega^{(1)}_{t,\epsilon_{1.6},\epsilon_{1.7},i\bar j}dz_i^{(1)}\wedge d\bar z_j^{(1)}\wedge(\sum_{\alpha=1}^{n}\frac{1}{\lambda_\alpha})(\omega_{t,\epsilon_{1.6},\epsilon_{1.7}}^{(2)})^{n},
\end{split}\] and
\[\begin{split}
&n(n-1)\omega_{t,\epsilon_{1.6},\epsilon_{1.7}}^{(1,2)}\wedge\omega_{t,\epsilon_{1.6},\epsilon_{1.7}}^{(2,1)}\wedge(\omega_{t,\epsilon_{1.6},\epsilon_{1.7}}^{(2)})^{n-2}\wedge\pi_2^*\chi\\
&=-\sum_{i,j,k=1}^{n}\sqrt{-1}\omega^{(1,2)}_{t,\epsilon_{1.6},\epsilon_{1.7},i\bar k}\overline{\omega^{(1,2)}_{t,\epsilon_{1.6},\epsilon_{1.7},j\bar k}}dz_i^{(1)}\wedge d\bar z_j^{(1)}\wedge\frac{1}{\lambda_k}(\sum_{\alpha\not=k}\frac{1}{\lambda_\alpha})(\omega_{t,\epsilon_{1.6},\epsilon_{1.7}}^{(2)})^{n}\\
&\ge-\sum_{i,j,k=1}^{n}\sqrt{-1}\omega^{(1,2)}_{t,\epsilon_{1.6},\epsilon_{1.7},i\bar k}\overline{\omega^{(1,2)}_{t,\epsilon_{1.6},\epsilon_{1.7},j\bar k}}dz_i^{(1)}\wedge d\bar z_j^{(1)}\wedge\frac{1}{\lambda_k}(\sum_{\alpha=1}^{n}\frac{1}{\lambda_\alpha})(\omega_{t,\epsilon_{1.6},\epsilon_{1.7}}^{(2)})^{n}.
\end{split}\]
By Lemma \ref{Subadditive-submatrix},
\[\begin{split}
&P_{\pi_1^*\chi}(\sqrt{-1}\sum_{i,j=1}^{n}((\omega^{(1)}_{t,\epsilon_{1.6},\epsilon_{1.7},i\bar j}-\sum_{k=1}^{n}\frac{1}{\lambda_k}\omega^{(1,2)}_{t,\epsilon_{1.6},\epsilon_{1.7},i\bar k}\overline{\omega^{(1,2)}_{t,\epsilon_{1.6},\epsilon_{1.7},j\bar k}})dz_i^{(1)}\wedge d\bar z_j^{(1)}))\\
&\le P_{\chi_{M\times M}}(\omega_{t,\epsilon_{1.6},\epsilon_{1.7}})-\mathrm{tr}_{\omega^{(2)}_{t,\epsilon_{1.6},\epsilon_{1.7}}}(\pi_2^*\chi)\\
&\le (n+1)c-\mathrm{tr}_{\omega^{(2)}_{t,\epsilon_{1.6},\epsilon_{1.7}}}(\pi_2^*\chi).
\end{split}\]

Now we view \[\frac{c^{n-1}}{\int_M n\chi^n}(\mathrm{tr}_{\omega^{(2)}_{t,\epsilon_{1.6},\epsilon_{1.7}}}\pi_2^*\chi)(\omega^{(2)}_{t,\epsilon_{1.6},\epsilon_{1.7}})^{n}\] as a measure on $\{x_1\}\times M$, then it is easy to see that
\[\begin{split}
&\frac{c^{n-1}}{\int_M n\chi^n}\int_{\{x_1\}\times M}(\mathrm{tr}_{\omega^{(2)}_{t,\epsilon_{1.6},\epsilon_{1.7}}}\pi_2^*\chi)(\omega^{(2)}_{t,\epsilon_{1.6},\epsilon_{1.7}})^{n}\\
&=\frac{c^{n-1}}{\int_M \chi^n}\int_{\{x_1\}\times M}(\pi_2^*\chi)\wedge(\omega^{(2)}_{t,\epsilon_{1.6},\epsilon_{1.7}})^{n-1}\\
&=\frac{c^{n-1}}{\int_M \chi^n}\int_M\chi\wedge(\frac{\chi}{c})^{n-1}\\
&=1.
\end{split}\]

By the monotonicity and convexity of $P_\chi$, \[\begin{split}
P_\chi(\omega_{1,t,\epsilon_{1.6},\epsilon_{1.7}})&\le(n+1)c-\frac{c^{n-1}}{\int_M n\chi^n}\int_{\{x_1\}\times M}(\mathrm{tr}_{\omega^{(2)}_{t,\epsilon_{1.6},\epsilon_{1.7}}}\pi_2^*\chi)^2(\omega^{(2)}_{t,\epsilon_{1.6},\epsilon_{1.7}})^{n}\\
&\le(n+1)c-\frac{c^{n-1}}{\int_M n\chi^n}\frac{(\int_{\{x_1\}\times M}(\mathrm{tr}_{\omega^{(2)}_{t,\epsilon_{1.6},\epsilon_{1.7}}}\pi_2^*\chi)(\omega^{(2)}_{t,\epsilon_{1.6},\epsilon_{1.7}})^{n})^2}{\int_{\{x_1\}\times M}(\omega^{(2)}_{t,\epsilon_{1.6},\epsilon_{1.7}})^{n}}\\
&=(n+1)c-\frac{nc^{n-1}}{\int_M \chi^n}\frac{(\int_{M}(\frac{\chi}{c})^{n-1}\wedge\chi)^2}{\int_{M}(\frac{\chi}{c})^{n}}\\
&=c.
\end{split}\]

Up to here, we have not used the equation \[\mathrm{tr}_{\omega_{t,\epsilon_{1.6},\epsilon_{1.7}}}\chi_{M\times M}+f_{t,\epsilon_{1.6},\epsilon_{1.7}}\frac{\chi_{M\times M}^{2n}}{\omega_{t,\epsilon_{1.6},\epsilon_{1.7}}^{2n}}=(n+1)c.\] By the equation, $\omega_{t,\epsilon_{1.6},\epsilon_{1.7}}^{2n}\ge \frac{f_{t,\epsilon_{1.6},\epsilon_{1.7}}}{(n+1)c}\chi_{M\times M}^{2n}$. So as in Proposition 2.6 of \cite{DemaillyPaun}, it is easy to see that for any weak limit $\Theta$ of $\omega_{t,\epsilon_{1.6},\epsilon_{1.7}}^n$ when $t$ and $\epsilon_{1.6}$ converging to 0, $1_\Delta\Theta=\epsilon_{3.1}[\Delta]$ for a constant $\epsilon_{3.1}>0$. Let $\Delta_{\epsilon_{3.2}}$ be the $\epsilon_{3.2}$ neighborhood of $\Delta$ with respect to $\chi_{M\times M}$. Then for any $\delta>0$, for any small enough $\epsilon_{3.2}$ and $\epsilon_{3.3}$, the smoothing of \[\frac{c^{n-1}}{\int_M n\chi^n}\int_{(\{x_1\}\times M)\cap\Delta_{\epsilon_{3.2}}}\omega_{t,\epsilon_{1.6},\epsilon_{1.7}}^n\wedge\pi_2^*\chi-\frac{c^{n-1}}{\int_M n\chi^n}\epsilon_{3.1}\chi\] is at least $-\epsilon_{3.3}\chi$ for small enough $t$ and $\epsilon_{1.6}$.

Similarly, locally for any $n+1$ dimensional subvariety $V$ containing $\Delta$, for any weak limit $\Theta'$ of $\omega_{t,\epsilon_{1.6},\epsilon_{1.7}}^{n-1}$, $1_V\Theta'=\epsilon_{3.4}[V]$ for $\epsilon_{3.4}\ge 0$. Since the dimension of $\Delta$ is strictly smaller then $V$, for any fixed smoothing function, for any $\epsilon_{3.3}>0$, there exists $\epsilon_{3.2}>0$ such that the smoothing of \[\frac{c^{n-1}}{\int_M n\chi^n}\int_{(\{x_1\}\times M)\cap\Delta_{\epsilon_{3.2}}}\pi_1^*\chi((\omega^{(2)}_{t,\epsilon_{1.6},\epsilon_{1.7}})^{n-1}\wedge\pi_2^*\chi)\] is at most $\epsilon_{3.3}\chi$ for small enough $t$ and $\epsilon_{1.6}$.

Now let $\epsilon_{3.5}$ be an arbitrary small positive number. Then we can choose $\epsilon_{3.3}$ such that $\frac{\epsilon_{3.3}}{\epsilon_{3.5}}+\epsilon_{3.3}<\frac{1}{2}\frac{c^{n-1}}{\int_M n\chi^n}\epsilon_{3.1}$. Then we choose the number $\epsilon_{3.2}$ depending on $\epsilon_{3.3}$. For any K\"ahler form $\omega$ restricted to the first $n$ coordinates of $M\times M$, after choosing a good coordinate, assume that $\pi_1^*\chi=\delta_{i\bar j}$ and $\omega_{i\bar j}=\lambda_i\delta_{i\bar j}$. We define the truncation $T_{\frac{
\pi_1^*\chi}{\epsilon_{3.5}}}(\omega)$ by $(T_{\frac{\pi_1^*\chi}{\epsilon_{3.5}}}(\omega))_{i\bar j}=\min\{\lambda_i,\frac{1}{\epsilon_{3.5}}\}\delta_{i\bar j}$. Now consider the truncation $\omega_{1,t,\epsilon_{1.6},\epsilon_{1.7}}^{(\frac{\pi_1^*\chi}{\epsilon_{3.5}})}$ defined as \[\frac{c^{n-1}}{\int_M n\chi^n}(\pi_1)_*(T_{\frac{\pi_1^*\chi}{\epsilon_{3.5}}}(\tilde\omega^{(1)}_{t,\epsilon_{1.6},\epsilon_{1.7}})\wedge(\omega^{(2)}_{t,\epsilon_{1.6},\epsilon_{1.7}})^{n-1}\wedge\pi_2^*\chi),\]
where the (1,1)-form $\tilde\omega^{(1)}_{t,\epsilon_{1.6},\epsilon_{1.7}}$ is defined by
\[\begin{split}
&n\omega_{t,\epsilon_{1.6},\epsilon_{1.7}}^{(1)}\wedge(\omega_{t,\epsilon_{1.6},\epsilon_{1.7}}^{(2)})^{n-1}\wedge\pi_2^*\chi\\
&+n(n-1)\omega_{t,\epsilon_{1.6},\epsilon_{1.7}}^{(1,2)}\wedge\omega_{t,\epsilon_{1.6},\epsilon_{1.7}}^{(2,1)}\wedge(\omega_{t,\epsilon_{1.6},\epsilon_{1.7}}^{(2)})^{n-2}\wedge\pi_2^*\chi\\
&=\tilde\omega^{(1)}_{t,\epsilon_{1.6},\epsilon_{1.7}}\wedge(\omega^{(2)}_{t,\epsilon_{1.6},\epsilon_{1.7}})^{n-1}\wedge\pi_2^*\chi.
\end{split}\]
The smoothing of \[\omega_{1,t,\epsilon_{1.6},\epsilon_{1.7}}-\omega_{1,t,\epsilon_{1.6},\epsilon_{1.7}}^{(\frac{\pi_1^*\chi}{\epsilon_{3.5}})}-\frac{c^{n-1}}{\int_M n\chi^n}\epsilon_{3.1}\chi\] is at least $-\frac{\epsilon_{3.3}}{\epsilon_{3.5}}\chi-\epsilon_{3.3}\chi$ for small enough $t$ and $\epsilon_{1.6}$. In fact, the sum of the first two terms is non-negative outside $\Delta_{3.2}$, the sum of the first and third term inside $\Delta_{3.2}$ is at least $-\epsilon_{3.3}\chi$ and the second term inside $\Delta_{3.2}$ is at least $-\frac{\epsilon_{3.3}}{\epsilon_{3.5}}\chi$.

By the choice of $\epsilon_{3.3}$, the smoothing of \[\omega_{1,t,\epsilon_{1.6},\epsilon_{1.7}}-\omega_{1,t,\epsilon_{1.6},\epsilon_{1.7}}^{(\frac{\pi_1^*\chi}{\epsilon_{3.5}})}-\frac{1}{2}\frac{c^{n-1}}{\int_M n\chi^n}\epsilon_{3.1}\chi\] is nonnegative for small enough $t$ and $\epsilon_{1.6}$. On the other hands, \[P_{\pi_1^*\chi}(T_{\frac{\pi_1^*\chi}{\epsilon_{3.5}}}(\omega))-P_{\pi_1^*\chi}(\omega)\le (n-1)\epsilon_{3.5}\] for any (1,1)-form $\omega$ on the first $n$ coordinates of $M\times M$. So using the estimate of $P_\chi(\omega_{1,t,\epsilon_{1.6},\epsilon_{1.7}})$, it is easy to see that \[P_\chi(\omega_{1,t,\epsilon_{1.6},\epsilon_{1.7}}^{(\frac{\pi_1^*\chi}{\epsilon_{3.5}})})\le c+(n-1)\epsilon_{3.5}.\] So if $\chi\ge\chi_0$ on the support of the smoothing function for a K\"ahler form $\chi_0$ with constant coefficients, then $P_{\chi_0}$ acting on the smoothing of $\omega_{1,t,\epsilon_{1.6},\epsilon_{1.7}}^{(\frac{\pi_1^*\chi}{\epsilon_{3.5}})}$ is at most $c+(n-1)\epsilon_{3.5}$. So $P_{\chi_0}$ acting on the smoothing of $\omega_{1,t,\epsilon_{1.6},\epsilon_{1.7}}-\frac{1}{2}\frac{c^{n-1}}{\int_M n\chi^n}\epsilon_{3.1}\chi$ is also at most $c+(n-1)\epsilon_{3.5}$. Let $\omega_{1.5}$ be a weak limit of $\omega_{1,t,\epsilon_{1.6},\epsilon_{1.7}}-\frac{1}{2}\frac{c^{n-1}}{\int_M n\chi^n}\epsilon_{3.1}\chi$, then $\omega_{1.5}\in[\omega_0-\epsilon_{1.4}\chi]$, where $\epsilon_{1.4}=\frac{1}{2}\frac{c^{n-1}}{\int_M n\chi^n}\epsilon_{3.1}$. Moreover, $P_{\chi_0}$ acting on the smoothing of $\omega_{1.5}$ is at most $c+(n-1)\epsilon_{3.5}$. Since $\epsilon_{3.5}$ is arbitrary, it is at most $c$. This completes the proof of Theorem \ref{Current-sub-solution}.

\section{Regularization}

In this section, we prove Theorem \ref{Main-induction}. By Remark \ref{Base-case}, the $n=1$ and $n=2$ cases have been proved. By induction, we can assume that Theorem \ref{Main-induction} has been proved in dimension 1, 2, ..., $n-1$. By Section 1, we can in addition assume that the condition of Theorem \ref{Current-sub-solution} are satisfied. So by Theorem \ref{Current-sub-solution}, there exist a constant $\epsilon_{1.4}>0$ and a current $\omega_{1.5}\in[\omega_0-\epsilon_{1.4}\chi]$ such that \[c\omega_{1.5}^{n-1}-(n-1)\chi\wedge\omega_{1.5}^{n-2}\ge0\] in the sense of Definition \ref{Definition-current-sub-solution}.

Pick small enough  $\epsilon_{4.1}<\frac{1}{10000}$ such that \[\omega_0-100\epsilon_{4.1}\omega_0\ge(1+\epsilon_{4.1})^2(\omega_0-\epsilon_{1.4}\chi).\] Then there exists a current $\omega_{4.2}=\omega_0-100\epsilon_{4.1}\omega_0+\sqrt{-1}\partial\bar\partial\varphi_{4.2}$ such that \[\frac{c}{(1+\epsilon_{4.1})^2}\omega_{4.2}^{n-1}-(n-1)\chi\wedge\omega_{4.2}^{n-2}\ge0\] in the sense of Definition \ref{Definition-current-sub-solution}.

Now we pick a finite number of coordinate balls $B_{2r}(x_i)$ such that $B_r(x_i)$ is a cover of $M$. Moreover, we require that \[\chi_0^i<\chi<(1+\epsilon_{4.1})\chi_0^i\] on $B_{2r}(x_i)$ for K\"ahler forms $\chi_0^i$ with constant coefficients. We also assume that \[(1-\epsilon_{4.1})\sqrt{-1}\partial\bar\partial|z|^2\le\omega_0\le(1+\epsilon_{4.1})\sqrt{-1}\partial\bar\partial|z|^2\] on $B_{2r}(x_i)$. Let $\varphi_{\omega_0}^i$ be potential such that $\sqrt{-1}\partial\bar\partial\varphi_{\omega_0}^i=\omega_0$ on $B_{2r}(x_i)$. Then we also assume that \[|\varphi_{\omega_0}^i-|z|^2|\le\epsilon_{4.1} r^2.\] Let $\varphi_\delta^i$ be the smoothing of $\varphi_{4.2}+(1-100\epsilon_{4.1})\varphi_{\omega_0}^i$. When $\delta<\frac{r}{5}$, it is well defined on $\overline{B_{\frac{9}{5}r}(x_i)}$. By assumption, it is easy to see that \[\frac{c}{(1+\epsilon_{4.1})^2}(\sqrt{-1}\partial\bar\partial\varphi_\delta^i)^{n-1}-(n-1)\chi_0^i\wedge(\sqrt{-1}\partial\bar\partial\varphi_\delta^i)^{n-2}\ge0.\]
So  \[\frac{c}{1+\epsilon_{4.1}}(\sqrt{-1}\partial\bar\partial\varphi_\delta^i)^{n-1}-(n-1)\chi\wedge(\sqrt{-1}\partial\bar\partial\varphi_\delta^i)^{n-2}>0.\] Now define the function $\varphi_{4.3}^i$ from $\overline{B_{\frac{9}{5}r}(x_i)}$ to $\mathbb{R}$ as $\varphi_\delta^i-\varphi_{\omega_0}^i$, our goal is to show that for any $x\in M$, \[\epsilon_{4.1}r^2+\max_{\{i:x\in \overline{B_{\frac{9}{5}r}(x_i)}\setminus B_{\frac{8}{5}r}(x_i)\}}\varphi_{4.3}^i(x)<\max_{\{i:x\in \overline{B_{r}(x_i)}\}}\varphi_{4.3}^i(x).\] If this is true, then for the smooth function $\varphi_{4.4}=\tilde\max\varphi_{4.3}^i$ on $M$, where $\tilde\max$ means the regularized maximum by choosing the parameters ``$\eta_i$" in Lemma I.5.18 of \cite{Demailly} to be smaller than $\frac{\epsilon_{4.1}r^2}{3}$, the K\"ahler form $\omega_{4.4}=\omega_0+\sqrt{-1}\partial\bar\partial\varphi_{4.4}$ will satisfy $c\omega_{4.4}^{n-1}-(n-1)\chi\wedge\omega_{4.4}^{n-2}>0$. In general, this is not true. However, using the proof of the results of B\l{}ocki and Ko\l{}odziej \cite{BlockiKolodziej}, it is in fact true if the Lelong number is small enough. The details consist the rest of this section.

It is easy to see that if $x\in\overline{B_{\frac{9}{5}r}(x_i)}\cap\overline{B_{\frac{9}{5}r}(x_j)}$ and $\delta<\frac{r}{10}$, $B_{\frac{\delta}{2}}^i(x)\subset B_{\delta}^j(x)$. For any $\delta<\frac{r}{20}$ and $x\in\overline{B_{\frac{9}{5}r}(x_i)}$, we define $\hat\varphi_\delta^i$ by \[\hat\varphi_\delta^i(x)=\sup_{B_\delta^i(x)}(\varphi_{4.2}+(1-100\epsilon_{4.1})\varphi_{\omega_0}^i)\] and define $\nu^i(x,\delta)$ by \[\nu^i(x,\delta)=\frac{\hat\varphi_{\frac{r}{16}}^i(x)-\hat\varphi_\delta^i(x)}{\log(\frac{r}{16})-\log\delta}.\] Then $\nu^i(x,\delta)$ is monotonically non-decreasing in $\delta$. Recall that the Lelong number is defined by \[\nu^i(x)=\lim_{\delta\rightarrow0}\nu^i(x,\delta).\] It is independent of $i$ and can be denoted as $\nu(x)$ instead. Recall the definition of $\rho$ in Definition \ref{Definition-smoothing}. Let \[\epsilon_{4.5}=\frac{\epsilon_{4.1} r^2}{5(\int_0^1\log(\frac{1}{t})\mathrm{Vol}(\partial B_1(0))t^{2n-1}\rho(t)dt+\log 2+\frac{3^{2n-1}}{2^{2n-3}}\log 2)},\]
then by the result of Siu \cite{Siu}, the set $Y=\{x:\nu(x)\ge\epsilon_{4.5}\}$ is a subvariety.

For simplicity, we assume that $Y$ is smooth. The singular case will be done at the end of this section.

Since $Y$ is smooth by our assumption, as in Section 1, there exists a smooth function $\varphi_{1.11}$ in a neighborhood $O$ of $Y$ such that \[(c-\frac{n-p}{2}\epsilon_{1.1})\omega_{1.11}^{n-1}-(n-1)\chi\wedge\omega_{1.11}^{n-2}>0\] on $O$. Now we pick smaller neighborhoods $O'$ and $O''$ such that $\overline{O'}\subset O$ and $\overline{O''}\subset O'$. We need to prove the following proposition:
\begin{proposition}
(1) For small enough $\delta<\frac{r}{20}$, if \[\max_{\{i:x\in \overline{B_{\frac{9}{5}r}(x_i)}\}}\nu^i(x,\delta)\le2\epsilon_{4.5},\] then \[\sup_{\overline{O'}}\varphi_{1.11}+3\epsilon_{4.5}\log\delta+\epsilon_{4.1}r^2<\max_{\{i:x\in \overline{B_{\frac{9}{5}r}(x_i)}\}}(\varphi_\delta^i(x)-\varphi_{\omega_0}^i(x)).\]

(2) For small enough $\delta<\frac{r}{20}$, if \[\inf_{\overline{O'}}\varphi_{1.11}+3\epsilon_{4.5}\log\delta-\epsilon_{4.1}r^2\le\max_{\{i:x\in \overline{B_{\frac{9}{5}r}(x_i)}\}}(\varphi_\delta^i(x)-\varphi_{\omega_0}^i(x)),\] then \[\max_{\{i:x\in \overline{B_{\frac{9}{5}r}(x_i)}\}}\nu^i(x,\delta)<4\epsilon_{4.5}.\]

(3) For small enough $\delta<\frac{r}{20}$, if \[\max_{\{i:x\in \overline{B_{\frac{9}{5}r}(x_i)}\}}\nu^i(x,\delta)\le4\epsilon_{4.5}.\] then \[\max_{\{i:x\in \overline{B_{\frac{9}{5}r}(x_i)}\setminus B_{\frac{8}{5}r}(x_i)\}}(\varphi_\delta^i(x)-\varphi_{\omega_0}^i(x))+\epsilon_{4.1}r^2<\max_{\{i:x\in \overline{B_{r}(x_i)}\}}(\varphi_\delta^i(x)-\varphi_{\omega_0}^i(x)).\]
\label{Smoothing-global-error-control}
\end{proposition}

If Proposition \ref{Smoothing-global-error-control} is true, for small enough $\delta$, we can define $\varphi_{1.12}$ as the regularized maximum of $\varphi_{1.11}(x)+3\epsilon_{4.5}\log\delta$ on $\overline{O'}$ and $\varphi_\delta^i-\varphi_{\omega_0}^i$ on $\overline{B_{\frac{9}{5}r}(x_i)}$. Since $\nu(x)<\epsilon_{4.5}$ for $x\not\in Y$, for small enough $\delta$, $\max_{\{i:x\in \overline{B_{\frac{9}{5}r}(x_i)}\}}\nu^i(x,\delta)\le2\epsilon_{4.5}$ for all $x\not\in O''$. So by Proposition \ref{Smoothing-global-error-control} (1), we do not need to worry about the discontiuty near the boundary of $\overline{O'}$. By Proposition \ref{Smoothing-global-error-control} (2) and (3), there is also no need to worry about the discontinuity near the boundary of $\overline{B_{\frac{9}{5}r}(x_i)}$. In conclusion, $\varphi_{1.12}$ will be smooth and satisfy \[c\omega_{1.12}^{n-1}-(n-1)\chi\wedge\omega_{1.12}^{n-2}>0\] on $M$ as long as $Y$ is smooth and Proposition \ref{Smoothing-global-error-control} is true.

In order to prove Proposition \ref{Smoothing-global-error-control}, we need the following lemma of B\l{}ocki and Ko\l{}odziej \cite{BlockiKolodziej}.
\begin{lemma}
For any $\delta<\frac{r}{20}$ and $x\in\overline{B_{\frac{9}{5}r}(x_i)}$, the following estimates hold:

(1) $0\le\hat\varphi_\delta^i-\hat\varphi_{\frac{\delta}{a}}^i\le\nu^i(x,\delta)\log a$ for all $a\ge1$,

(2) $0\le\hat\varphi_\delta^i-\varphi_\delta^i\le\nu^i(x,\delta)(\int_0^1\log(\frac{1}{t})\mathrm{Vol}(\partial B_1(0))t^{2n-1}\rho(t)dt+\frac{3^{2n-1}}{2^{2n-3}}\log 2)$.
\label{Smoothing-local-error-control}
\end{lemma}
\begin{proof}
For readers' convenience, we almost line by line copy the paper \cite{BlockiKolodziej} here:

(1) It follows from the logarithmical convexity of $\hat\varphi_\delta^i$ and the definition of $\nu^i(x,\delta)$.

(2) Define another regularization $\tilde\varphi_\delta^i$ by
\[\tilde\varphi_\delta^i(x)=\frac{1}{\mathrm{Vol}(\partial B_\delta(x))}\int_{\partial B_\delta(x)}(\varphi_{4.2}+(1-100\epsilon_{4.1})\varphi_{\omega_0}^i)d\mathrm{Vol}.\] Then by the Poisson kernel for subharmonic functions \cite{BlockiKolodziej} and the estimate in (1),
\[\hat\varphi_{t\delta}^i(x)-\tilde\varphi_{t\delta}^i(x)\le\frac{3^{2n-1}}{2^{2n-2}}(\hat\varphi_{t\delta}^i-\hat\varphi_{t\delta/2}^i)\le(\frac{3^{2n-1}}{2^{2n-2}}\log 2)\nu^i(x,t\delta)\] for all $t\in (0,1]$. By monotonicity, \[\hat\varphi_{t\delta}^i(x)-\tilde\varphi_{t\delta}^i(x)\le(\frac{3^{2n-1}}{2^{2n-2}}\log 2)\nu^i(x,t\delta)\le(\frac{3^{2n-1}}{2^{2n-2}}\log 2)\nu^i(x,\delta).\] Define \[\tilde\rho(t)=\mathrm{Vol}(\partial B_1(0))t^{2n-1}\rho(t),\] then $\int_0^1\tilde\rho(t)=1$.
So
\[\tilde\varphi_\delta^i-\varphi_\delta^i=\int_0^1(\tilde\varphi_\delta^i-\tilde\varphi_{t\delta}^i)\tilde\rho(t)dt
\le\int_0^1(\hat\varphi_\delta^i-\hat\varphi_{t\delta}^i)\tilde\rho(t)dt+(\frac{3^{2n-1}}{2^{2n-2}}\log 2)\nu^i(x,\delta).\]
By the estimate in (1) again, \[\hat\varphi_\delta^i-\hat\varphi_{t\delta}^i\le\nu^i(x,\delta)\log(\frac{1}{t}).\]
The other side of inequality $0\le\hat\varphi_\delta^i-\varphi_\delta^i$ is trivial.
\end{proof}

It is easy to see that there exists a constant $C_{4.6}$ such that for any $\delta<\frac{r}{20}$ and $x\in\overline{B_{\frac{9}{5}r}(x_i)}$, $\nu^i(x,\delta)<C_{4.6}$. Now we are ready to prove Proposition \ref{Smoothing-global-error-control}.

(1) Suppose $\delta<\frac{r}{20}$, $x\in\overline{B_{\frac{9}{5}r}(x_i)}$ and \[\nu^i(x,\delta)=\frac{\hat\varphi_{\frac{r}{16}}^i(x)-\hat\varphi_\delta^i(x)}{\log(\frac{r}{16})-\log\delta}\le2\epsilon_{4.5},\] then \[\hat\varphi_\delta^i(x)\ge\hat\varphi_{\frac{r}{16}}^i(x)+2\epsilon_{4.5}(\log\delta-\log(\frac{r}{16}))\ge-C_{4.7}+2\epsilon_{4.5}\log\delta.\] By Lemma \ref{Smoothing-local-error-control} (2), \[\varphi_\delta^i(x)\ge-C_{4.8}+2\epsilon_{4.5}\log\delta.\] It is easy to see that for $\delta$ small enough, \[\sup_{\overline{O'}}\varphi_{1.11}+3\epsilon_{4.5}\log\delta+\epsilon_{4.1}r^2<\varphi_\delta^i(x)-\varphi_{\omega_0}^i(x)\] because $\varphi_{\omega_0}^i$ is uniformly bounded on $\overline{B_{\frac{9}{5}r}(x_i)}$.

(2) Suppose $\delta<\frac{r}{20}$, $x\in\overline{B_{\frac{9}{5}r}(x_i)}$ and \[\inf_{\overline{O'}}\varphi_{1.11}+3\epsilon_{4.5}\log\delta-\epsilon_{4.1}r^2\le\varphi_\delta^i(x)-\varphi_{\omega_0}^i(x),\] then as before \[\hat\varphi_\delta^i(x)\ge\varphi_\delta^i(x)\ge-C_{4.9}+3\epsilon_{4.5}\log\delta.\] By Lemma \ref{Smoothing-local-error-control} (1) and the definition of $\hat\varphi_{\frac{\delta}{2}}^i(x)$, \[\sup_{B_{\frac{\delta}{2}}^i(x)}\varphi_{4.2}\ge-C_{4.10}+3\epsilon_{4.5}\log\delta.\] If $x\in\overline{B_{\frac{9}{5}r}(x_j)}$, then $B_{\frac{\delta}{2}}^i(x)\subset B_{\delta}^j(x)$ and therefore \[\sup_{B_\delta^j(x)}\varphi_{4.2}\ge\sup_{B_{\frac{\delta}{2}}^i(x)}\varphi_{4.2}\ge-C_{4.10}+3\epsilon_{4.5}\log\delta.\] By the definition of $\hat\varphi_\delta^j(x)$ and $\nu^j(x,\delta)$, it is easy to see that $\nu^j(x,\delta)<4\epsilon_{4.5}$ if $\delta$ is small enough.

(3) Suppose $\delta<\frac{r}{20}$, $x\in(\overline{B_{\frac{9}{5}r}(x_i)}\setminus B_{\frac{8}{5}r}(x_i))\cap\overline{B_{r}(x_j)}$ and \[\max_{\{i:x\in \overline{B_{\frac{9}{5}r}(x_i)}\}}\nu^i(x,\delta)\le4\epsilon_{4.5},\] then \[\hat\varphi_{\frac{\delta}{2}}^i(x)-\varphi_{\omega_0}^i(x)\le\sup_{B_{\frac{\delta}{2}}^i(x)}\varphi_{4.2}+2\epsilon_{4.1}r^2+(2r+\delta)^2-(2r)^2-100\epsilon_{4.1}(\frac{7}{5}r)^2,\]
and \[\sup_{B_\delta^j(x)}\varphi_{4.2}\le\hat\varphi_\delta^j(x)-\varphi_{\omega_0}^j(x)+2\epsilon_{4.1}r^2+(2r+\delta)^2-(2r)^2+100\epsilon_{4.1}(\frac{6}{5}r)^2.\]
By Lemma \ref{Smoothing-local-error-control} (1),
\[\hat\varphi_\delta^i-\hat\varphi_{\frac{\delta}{2}}^i\le\nu^i(x,\delta)\log 2\le4\epsilon_{4.5}\log 2.\]
By Lemma \ref{Smoothing-local-error-control} (2), $\varphi_\delta^i\le \hat\varphi_\delta^i$ and
\[\hat\varphi_\delta^j-\varphi_\delta^j\le4\epsilon_{4.5}(\int_0^1\log(\frac{1}{t})\mathrm{Vol}(\partial B_1(0))t^{2n-1}\rho(t)dt+\frac{3^{2n-1}}{2^{2n-3}}\log 2).\]
Since $\sup_{B_{\frac{\delta}{2}}^i(x)}\varphi_{4.2}\le\sup_{B_\delta^j(x)}\varphi_{4.2}$, by summing everything together, for $\delta$ small enough, $\varphi_\delta^i(x)-\varphi_{\omega_0}^i(x)+\epsilon_{4.1}r^2<\varphi_\delta^j(x)-\varphi_{\omega_0}^j(x)$. We are done if $Y$ is smooth.

In general $Y$ is singular. By Hironaka's desingularization theorem, there exists a blow-up $\tilde M$ of $M$ obtained by a sequence of blow-ups with smooth centers such that the proper transform $\tilde Y$ of $Y$ is smooth. Without loss of generality, assume that we only need to blow up once. Let $\pi$ be the projection of $\tilde M$ to $M$. Let $E$ be the exceptional divisor. Let $s$ be the defining section of $E$. Let $h$ be any smooth metric on the line bundle $[E]$, then $\frac{\sqrt{-1}}{2\pi}\partial\bar\partial\log|s|_h^2=[E]+\omega_{4.11}$ by the Poincar\'e-Lelong equation. Then it is well known that the smooth (1,1)-form $\omega_{4.11}\in-[E]$ on $\tilde M$ and $\omega_{4.11}>-C_{4.12}\pi^*\omega_0$. For example, see Lemma 3.5 of \cite{DemaillyPaun} for the explanation. Define $\omega_{4.13}=C_{4.12}\pi^*\omega_0+\omega_{4.11}$. Then $\omega_{4.13}$ is a K\"ahler form on $\tilde M$.

\begin{lemma}
Let $C_{4.14}=\frac{6n}{\epsilon_{1.1}}$. Then for all small enough $t$ and $q$-dimensional subvarieties $V$ of $\tilde M$, as long as $q<n$, \[\begin{split}
&\int_V (c-\frac{n-q}{3n}\epsilon_{1.1})((1+C_{4.14}t)\pi^*\omega_0+C_{4.14}t^2\omega_{4.13})^q\\
&\ge\int_V q((1+C_{4.14}t)\pi^*\omega_0+C_{4.14}t^2\omega_{4.13})^{q-1}\wedge(\pi^*\chi+t^2\omega_{4.13}).
\end{split}\]
\label{Stability-on-blow-up}
\end{lemma}
\begin{proof}
By assumption,
\[\int_V (c-\frac{\epsilon_{1.1}}{3})\pi^*\omega_0^q-q\pi^*\omega_0^{q-1}\wedge\pi^*\chi=\int_{\pi(V)} (c-\frac{\epsilon_{1.1}}{3})\omega_0^q-q\omega_0^{q-1}\wedge\chi\ge0.\]
So \[\int_V (c-\frac{\epsilon_{1.1}}{3})((1+C_{4.14}t)\pi^*\omega_0)^q-q((1+C_{4.14}t)\pi^*\omega_0)^{q-1}\wedge((1+C_{4.14}t)\pi^*\chi)\ge0.\]
It suffices to show that \[\begin{split}
&\int_V (c-\frac{\epsilon_{1.1}}{3})((1+C_{4.14}t)\pi^*\omega_0+C_{4.14}t^2\omega_{4.13})^q\\
&-q((1+C_{4.14}t)\pi^*\omega_0+C_{4.14}t^2\omega_{4.13})^{q-1}\wedge(\pi^*\chi+t^2\omega_{4.13})\\
&\ge\int_V (c-\frac{\epsilon_{1.1}}{3})((1+C_{4.14}t)\pi^*\omega_0)^q\\
&-q((1+C_{4.14}t)\pi^*\omega_0)^{q-1}\wedge((1+C_{4.14}t)\pi^*\chi).
\end{split}\]

Since it only depends on the cohomology classes, we want to replace $\omega_0$ by a better representative in its cohomology class. Remark that $\pi(E)$ is smooth by assumption. So we can apply Theorem \ref{Main-induction} to $\pi(E)$. As in Section 1, there exists a smooth function $\varphi_{4.15}$ on a neighborhood $O_{4.16}$ of $\pi(E)$ in $M$ such that $\omega_{4.15}=\omega_0+\sqrt{-1}\partial\bar\partial\varphi_{4.15}$ satisfies \[(c-\frac{\epsilon_{1.1}}{2})\omega_{4.15}^{n-1}-(n-1)\chi\wedge\omega_{4.15}^{n-2}>0\] on $O_{4.16}$. Define $\varphi_{4.17}=\pi_*\frac{\log|s|_h^2}{4\pi C_{4.12}}$ on $M\setminus\pi(E)$. Recall the definition of the regularized maximum in Lemma I.5.18 of \cite{Demailly}. For large enough $C_{4.18}$, let $\varphi_{4.19}$ be the regularized maximum of $\varphi_{4.17}+C_{4.18}$ and $\varphi_{4.15}$. Then $\varphi_{4.19}$ is smooth on $M$ and $\omega_{4.19}=\omega_0+\sqrt{-1}\partial\bar\partial\varphi_{4.19}>0$ on $M$. Moreover, there exists a smaller neighborhood $O_{4.20}$ of $\pi(E)$ such that $\varphi_{4.19}=\varphi_{4.15}$ on $O_{4.20}\subset O_{4.16}$.

After replacing $\omega_0$ by $\omega_{4.19}$, it suffices to show that
\[\begin{split}
&(c-\frac{\epsilon_{1.1}}{3})\sum_{i=1}^{q}\frac{q!}{i!(q-i)!}((1+C_{4.14}t)\pi^*\omega_{4.19})^{q-i}\wedge(C_{4.14}t^2\omega_{4.13})^i\\
-&q\sum_{i=1}^{q}\frac{(q-1)!}{(q-i)!(i-1)!}((1+C_{4.14}t)\pi^*\omega_{4.19})^{q-i}\wedge C_{4.14}^{i-1}(t^2\omega_{4.13})^i\\
-&q\sum_{i=1}^{q-1}\frac{(q-1)!}{i!(q-1-i)!}((1+C_{4.14}t)\pi^*\omega_{4.19})^{q-1-i}\\
&\wedge(C_{4.14}t^2\omega_{4.13})^i\wedge(1+C_{4.14}t)\pi^*\chi\\
+&q((1+C_{4.14}t)\pi^*\omega_{4.19}+C_{4.14}t^2\omega_{4.13})^{q-1}\wedge C_{4.14}t\pi^*\chi\\
\ge&0.
\end{split}\]
By definition of $C_{4.14}$, \[q\frac{(q-1)!}{(q-i)!(i-1)!}<\frac{\epsilon_{1.1}}{6}\frac{q!}{i!(q-i)!}C_{4.14}\] for all $i=1, 2, ..., q$. So we can combine the first term and the second term.
If the point is inside $\pi^{-1}(O_{4.20})$, then for all $i=1, 2, ..., q-1$, \[(c-\frac{\epsilon_{1.1}}{2})(\pi^*\omega_{4.19})^{q-i}\ge(q-i)(\pi^*\omega_{4.19})^{q-1-i}\wedge\pi^*\chi\] because \[(c-\frac{\epsilon_{1.1}}{2})\omega_{4.19}^{n-1}\ge(n-1)\omega_{4.19}^{n-2}\wedge\chi\] on $O_{4.20}$. So the sum of the first three terms is non-negative if $i=1,2,...,q-1$. So we are done because the $i=q$ term and the fourth term are non-negative.
If the point is outside $\pi^{-1}(O_{4.20})$, then there exists $C_{4.21}$ such that \[C_{4.21}\pi^*\chi>\omega_{4.13}>C_{4.21}^{-1}\pi^*\chi\] and \[C_{4.21}\pi^*\chi>\pi^*\omega_{4.19}>C_{4.21}^{-1}\pi^*\chi\] on $\overline{\tilde M\setminus \pi^{-1}(O_{4.20})}$. The only first order term in $t$ is \[q\pi^*\omega_{4.19}^{q-1}\wedge C_{4.14}t\pi^*\chi.\] Since it is positive, for small enough $t$, we also get the required inequality.
\end{proof}

Now we pick $t>0$ such that $t$ satisfies Lemma \ref{Stability-on-blow-up} and \[\frac{c}{1+C_{4.14}t+C_{4.12}C_{4.14}t^2}>\max\{c-\frac{\epsilon_{1.1}}{4n},\frac{c}{1+\epsilon_{4.1}}\}.\]
We apply Theorem \ref{Main-induction} to the lower dimensional smooth manifold $\tilde Y$ with the K\"ahler forms $(1+C_{4.14}t)\pi^*\omega_0+C_{4.14}t^2\omega_{4.13}$ and $\pi^*\chi+t^2\omega_{4.13}$. As in Section 1, there exists a smooth function $\varphi_{4.22}$ on a neighborhood of $\tilde Y$ in $\tilde M$ such that \[\omega_{4.22}=(1+C_{4.14}t)\pi^*\omega_0+C_{4.14}t^2\omega_{4.13}+\sqrt{-1}\partial\bar\partial\varphi_{4.22}\] satisfies \[(c-\frac{\epsilon_{1.1}}{4n})\omega_{4.22}^{n-1}-(n-1)(\pi^*\chi+t^2\omega_{4.13})\wedge\omega_{4.22}^{n-2}>0\] near $\tilde Y$. Similarly, let $\varphi_{4.23}$ be the potential near $E$. For large enough constant $C_{4.24}$, define \[\varphi_{4.25}=\tilde\max\{\varphi_{4.23},\varphi_{4.22}+C_{4.24}^{-1}\pi^*\varphi_{4.17}+C_{4.24}\}\] and \[\omega_{4.25}=(1+C_{4.14}t)\pi^*\omega_0+C_{4.14}t^2\omega_{4.13}+\sqrt{-1}\partial\bar\partial\varphi_{4.25}.\]
Then \[(c-\frac{\epsilon_{1.1}}{4n})\omega_{4.25}^{n-1}-(n-1)(\pi^*\chi+t^2\omega_{4.13})\wedge\omega_{4.25}^{n-2}>0\] on a neighborhood $O$ of $\tilde Y\cup E$ in $\tilde M$. Since $t^2\omega_{4.13}>0$, it is easy to see that \[(c-\frac{\epsilon_{1.1}}{4n})(\pi_*\omega_{4.25})^{n-1}-(n-1)\chi\wedge(\pi_*\omega_{4.25})^{n-2}>0\] on $\pi(O\setminus E)$. Now we choose neighborhoods $O'$ and $O''$ of $Y\cup\pi(E)$ in $M$ such that $\overline{O'}\subset\pi(O)$ and $\overline{O''}\subset O'$. Then as before, for small enough $\delta$, we can define $\varphi_{4.26}$ as the regularized maximum of $\pi_*\varphi_{4.25}+3\epsilon_{4.5}\log\delta$ on $\overline{O'}\setminus\pi(E)$ and $\varphi_\delta^i-\varphi_{\omega_0}^i$ on $\overline{B_{\frac{9}{5}r}(x_i)}$. Then $\varphi_{4.26}$ is smooth and bounded on $M\setminus\pi(E)$. Moreover, for \[\begin{split}
\omega_{4.26}&=(1+C_{4.14}t)\omega_0+C_{4.14}t^2\pi_*\omega_{4.13}+\sqrt{-1}\partial\bar\partial\varphi_{4.26}\\
&=(1+C_{4.14}t+C_{4.12}C_{4.14}t^2)\omega_0+C_{4.14}t^2\pi_*\omega_{4.11}+\sqrt{-1}\partial\bar\partial\varphi_{4.26},
\end{split}\]
it is easy to see that \[(\max\{c-\frac{\epsilon_{1.1}}{4n},\frac{c}{1+\epsilon_{4.1}}\})\omega_{4.26}^{n-1}-(n-1)\chi\wedge\omega_{4.26}^{n-2}>0\] on $M\setminus\pi(E)$ because $C_{4.14}t\omega_0+C_{4.14}t^2\pi_*\omega_{4.13}>0$. Now we define \[\omega_{4.27}=\frac{\omega_{4.26}}{1+C_{4.14}t+C_{4.12}C_{4.14}t^2}=\omega_0+\frac{C_{4.14}t^2\pi_*\omega_{4.11}+\sqrt{-1}\partial\bar\partial\varphi_{4.26}}{1+C_{4.14}t+C_{4.12}C_{4.14}t^2},\] then by the choice of $t$,
\[c\omega_{4.27}^{n-1}-(n-1)\chi\wedge\omega_{4.27}^{n-2}>0\] on $M\setminus\pi(E)$. For large enough constant $C_{4.28}$, define \[\varphi_{4.29}=\tilde\max\{\frac{\frac{C_{4.14}t^2}{2\pi}\pi_*\log|s|_h^2+\varphi_{4.26}}{1+C_{4.14}t+C_{4.12}C_{4.14}t^2}+C_{4.28},\varphi_{4.15}\},\] then $\varphi_{4.29}$ is smooth on $M$ and $\omega_{4.29}=\omega_0+\sqrt{-1}\partial\bar\partial\varphi_{4.29}$ satisfies \[c\omega_{4.29}^{n-1}-(n-1)\chi\wedge\omega_{4.29}^{n-2}>0\] on $M$. We are done.

\section{Deformed Hermitian-Yang-Mills Equation}

In this section, we prove Theorem \ref{Deformed-Hermitian-Yang-Mills}.
The equation \[\sum_{i=1}^{n}\arctan\lambda_i=\hat\theta\] for eigenvalues $\lambda_i$ of $\omega_\varphi=\omega_0+\sqrt{-1}\partial\bar\partial\varphi>0$ with respect to $\chi$ is the same as the equation
\[\sum_{i=1}^{n}\arctan(\frac{1}{\lambda_i})=\frac{n\pi}{2}-\hat\theta.\]
To simplify the notations, define $\theta_0=\frac{n\pi}{2}-\hat\theta$.
Then the equation is equivalent to the inequality \[\sum_{i\not=j}\arctan(\frac{1}{\lambda_i})<\theta_0\] for $j=1, 2, ... ,n$ and the equation
\[\mathrm{Im}(\omega_\varphi+\sqrt{-1}\chi)^n=\tan(\theta_0)\mathrm{Re}(\omega_\varphi+\sqrt{-1}\chi)^n.\] Inspired by the work of Pingali in the toric case \cite{Pingali}, the analogy of Theorem \ref{Main-induction} is the following:

\begin{theorem}
Fix a K\"ahler manifold $M^n$ with K\"ahler metrics $\chi$ and $\omega_0$. Let $\theta_0\in(0,\frac{\pi}{4})$ be a constant and $f>-\frac{1}{100n}$ be a smooth function satisfying \[\int_M f\frac{\chi^n}{n!}=\int_M(\frac{\tan(\theta_0)\mathrm{Re}(\omega_\varphi+\sqrt{-1}\chi)^n}{n!}-\frac{\mathrm{Im}(\omega_\varphi+\sqrt{-1}\chi)^n}{n!})\ge 0,\] then there exists a smooth function $\varphi$ satisfying the equation
\[\mathrm{Im}(\omega_\varphi+\sqrt{-1}\chi)^n+f\chi^n=\tan(\theta_0)\mathrm{Re}(\omega_\varphi+\sqrt{-1}\chi)^n\] and the inequality \[\sum_{i\not=j}\arctan(\frac{1}{\lambda_i})<\theta_0\] for $j=1, 2, ... ,n$ and eigenvalues $\lambda_i$ of $\omega_\varphi=\omega_0+\sqrt{-1}\partial\bar\partial\varphi>0$ with respect to $\chi$ if there exists a constant $\epsilon_{1.1}>0$ and for all $p$-dimensional subvarieties $V$ with $p=1,2,...,n$ ($V$ can be chosen as $M$), there exist smooth functions $\theta_V(t)$ from $[1,\infty)$ to $[\frac{p\pi}{2}-\theta_0+(n-p)\epsilon_{1.1},\frac{p\pi}{2})$ such that for all $t\in[1,\infty)$,
\[\int_V(\chi+\sqrt{-1}t\omega_0)^p\not=0, \theta_V(t)=\mathrm{arg}(\int_V(\chi+\sqrt{-1}t\omega_0)^p), \lim_{t\rightarrow\infty}\theta_V(t)=\frac{p\pi}{2}.\]
\label{Induction-defomed-Hermitian-Yang-Mills}
\end{theorem}

When $n=1$, it is trivial. In higher dimensions, we need to prove it by induction.

Inspired by the work of Collins-Jacob-Yau \cite{CollinsJacobYau}, the analogy of Theorem \ref{Solvability-assuming-sub-solution} is the following:

\begin{theorem}
Fix a K\"ahler manifold $M^n$ with K\"ahler metrics $\chi$ and $\omega_0$. Let $\theta_0\in(0,\frac{\pi}{4})$ be a constant and $f>-\frac{1}{100n}$ be a smooth function satisfying \[\int_M f\frac{\chi^n}{n!}=\int_M(\frac{\tan(\theta_0)\mathrm{Re}(\omega_\varphi+\sqrt{-1}\chi)^n}{n!}-\frac{\mathrm{Im}(\omega_\varphi+\sqrt{-1}\chi)^n}{n!})\ge 0,\] then there exists a smooth function $\varphi$ satisfying the equation
\[\mathrm{Im}(\omega_\varphi+\sqrt{-1}\chi)^n+f\chi^n=\tan(\theta_0)\mathrm{Re}(\omega_\varphi+\sqrt{-1}\chi)^n\] and the inequality \[\sum_{i\not=j}\arctan(\frac{1}{\lambda_i})<\theta_0\] for $j=1, 2, ... ,n$ and eigenvalues $\lambda_i$ of $\omega_\varphi=\omega_0+\sqrt{-1}\partial\bar\partial\varphi>0$ with respect to $\chi$ if \[\sum_{i\not=j}\arctan(\frac{1}{\lambda_{i,0}})<\theta_0\] for $j=1, 2, ... ,n$ and eigenvalues $\lambda_{i,0}$ of $\omega_0>0$ with respect to $\chi$.
\label{Estimate-defomed-Hermitian-Yang-Mills}
\end{theorem}

We will use the continuity method three times to prove Theorem \ref{Estimate-defomed-Hermitian-Yang-Mills}. Let $\tilde\omega_0$ be the form $\omega_0$ and $\tilde f$ be the function $f$ in Theorem \ref{Estimate-defomed-Hermitian-Yang-Mills}. There exists a constant $C_{5.1}$ such that $\tilde\omega_0\ge C_{5.1}\chi$. We start from $f=0$ and $\omega_0=\cot(\frac{\theta_0}{n})\chi$. In this case, $\varphi=0$ is the solution. Then we let $\omega_0=t\cot(\frac{\theta_0}{n})\chi+(1-t)(C_{5.1}^{-1}\cot(\frac{\theta_0}{n})+1)\tilde\omega_0$ and $f$ be the non-negative constant satisfying the integrability condition as the first path. It will imply the result for $\omega_0=(C_{5.1}^{-1}\cot(\frac{\theta_0}{n})+1)\tilde\omega_0$. Then we let $\omega_0=t\tilde\omega_0$ and $f$ be the constant satisfying the integrability condition as the second path. $f$ must be non-negative because
\[\begin{split}
&\frac{d}{dt}\int_M(\frac{\tan(\theta_0)\mathrm{Re}(t\omega_0+\sqrt{-1}\chi)^n}{n!}-\frac{\mathrm{Im}(t\omega_0+\sqrt{-1}\chi)^n}{n!})\\
&=\int_M(\frac{\tan(\theta_0)\mathrm{Re}(t\omega_0+\sqrt{-1}\chi)^{n-1}}{(n-1)!}-\frac{\mathrm{Im}(t\omega_0+\sqrt{-1}\chi)^{n-1}}{(n-1)!})\wedge\omega_0>0
\end{split}\]
by the assumption on $\omega_0$ and Lemma 8.2 of \cite{CollinsJacobYau}. The continuity method will imply the result for $\tilde\omega_0$ when $f$ is the non-negative constant $f_0$ satisfying the integrability condition. Finally, we let $\omega_0=\tilde\omega_0$ and $f=t\tilde f+(1-t)f_0$ be the third continuity path. It will imply Theorem \ref{Estimate-defomed-Hermitian-Yang-Mills}.

It is easy to see the openness along the paths. Thus, we only need to prove the \textit{a priori} estimate along the paths. It will be achieved by Sz\'ekelyhidi's estimates in \cite{Szekelyhidi}. First of all, we need to rewrite the equation.

The equation \[\mathrm{Im}(\omega_\varphi+\sqrt{-1}\chi)^n+f\chi^n=\tan(\theta_0)\mathrm{Re}(\omega_\varphi+\sqrt{-1}\chi)^n\] can be written as
\[\mathrm{Im}\prod_{i=1}^{n}(\lambda_i+\sqrt{-1})+f=\tan(\theta_0)\mathrm{Re}\prod_{i=1}^{n}(\lambda_i+\sqrt{-1}).\]
It is equivalent to
\[\sin(\sum_{i=1}^{n}\arctan(\frac{1}{\lambda_i}))+\frac{f}{\prod_{i=1}^{n}\sqrt{\lambda_i^2+1}}=\tan(\theta_0)\cos(\sum_{i=1}^{n}\arctan(\frac{1}{\lambda_i})),\]
which is the same as \[\sin(\theta_0-\sum_{i=1}^{n}\arctan(\frac{1}{\lambda_i}))=\frac{f\cos(\theta_0)}{\prod_{i=1}^{n}\sqrt{\lambda_i^2+1}}.\]

Let $\Gamma$ be the region consisting of $(\lambda_1,...,\lambda_n)\in\mathbb{R}^n$ such that $\lambda_i>0$ and \[\sum_{i\not=j}\arctan(\frac{1}{\lambda_i})<\theta_0\] for all $j=1, 2, ... ,n$, then we want to study the function
\[F(f,\lambda_1,...,\lambda_n)=\sin(\theta_0-\sum_{i=1}^{n}\arctan(\frac{1}{\lambda_i}))-\frac{f\cos(\theta_0)}{\prod_{i=1}^{n}\sqrt{\lambda_i^2+1}}\] on $(-\frac{1}{100n},\infty)\times\bar\Gamma$, where $\bar\Gamma$ is the closure of $\Gamma$. For any K\"ahler form $\omega$, we say $\omega\in\Gamma_{\chi}$ if the eigenvalues of $\omega$ with respect to $\chi$ is in $\Gamma$. Similarly, we define $F_{\chi}(f,\omega)$ as $F(f,\lambda_i)$ for eigenvalues $\lambda_i$ of $\omega$ with respect to $\chi$.

In order to apply Sz\'ekelyhidi's estimates in \cite{Szekelyhidi}, we claim the following:

\begin{proposition}
Assume that $n\ge2$. If $f>-\frac{1}{100n}$, then

(1) $\frac{\partial}{\partial\lambda_i} F(f,\lambda)>0$ if $\lambda\in\Gamma$;

(2) $\frac{\partial^2}{\partial\lambda_i\partial\lambda_j} F(f,\lambda)\le -\cos(\theta_0)\frac{\lambda_i\delta_{ij}}{2(\lambda_i^2+1)^2}$ if $\lambda\in\Gamma$ and $F(f,\lambda)=0$;

(3) $\frac{\partial}{\partial\lambda_i} F(f,\lambda)\le\frac{\partial}{\partial\lambda_j} F(f,\lambda)$ if $\lambda\in \Gamma$ and $\lambda_i\ge\lambda_j$;

(4) If $\lambda\in\partial\Gamma$, then $F(f,\lambda)<0$;

(5) For any $\lambda\in\Gamma$, the set
\[\{\lambda'\in\Gamma: F(f,\lambda')=0, \lambda'_i\ge\lambda_i\}\] is bounded.
\label{Property-of-F}
\end{proposition}
\begin{proof}
(1) \[\frac{\partial}{\partial\lambda_i} F(f,\lambda)=\frac{\cos(\theta_0-\sum_{k=1}^{n}\arctan(\frac{1}{\lambda_k}))}{\lambda_i^2+1}+\frac{f\cos(\theta_0)}{\prod_{k=1}^{n}\sqrt{\lambda_k^2+1}}\frac{\lambda_i}{\lambda_i^2+1}.\]
Using the definition of $\Gamma$, it is easy to see that
\[0<\sum_{k=1}^{n}\arctan(\frac{1}{\lambda_k})<\frac{n\theta_0}{n-1}\] on $\Gamma$.
So \[\cos(\theta_0-\sum_{k=1}^{n}\arctan(\frac{1}{\lambda_k}))>\cos(\theta_0)>\frac{1}{\sqrt{2}}.\]
It is easy to see that $\frac{\partial}{\partial\lambda_i} F(f,\lambda)>0$ if $\lambda\in\Gamma$.

(2) \[\begin{split}
&\frac{\partial^2}{\partial\lambda_i\partial\lambda_j} F(f,\lambda)=-\cos(\theta_0-\sum_{k=1}^{n}\arctan(\frac{1}{\lambda_k}))\frac{2\lambda_i\delta_{ij}}{(\lambda_i^2+1)^2}\\
&-\sin(\theta_0-\sum_{k=1}^{n}\arctan(\frac{1}{\lambda_k}))\frac{1}{(\lambda_i^2+1)(\lambda_j^2+1)}\\
&-\frac{f\cos(\theta_0)}{\prod_{k=1}^{n}\sqrt{\lambda_k^2+1}}\frac{\lambda_i}{\lambda_i^2+1}\frac{\lambda_j}{\lambda_j^2+1}
+\frac{f\cos(\theta_0)}{\prod_{k=1}^{n}\sqrt{\lambda_k^2+1}}\frac{1-\lambda_i^2}{(\lambda_i^2+1)^2}\delta_{ij}.
\end{split}\]
The first term is at most $-\cos(\theta_0)\frac{2\lambda_i\delta_{ij}}{(\lambda_i^2+1)^2}$. When $F(f,\lambda)=0$, the second term equals to $-\frac{f\cos(\theta_0)}{\prod_{k=1}^{n}\sqrt{\lambda_k^2+1}}\frac{1}{(\lambda_i^2+1)(\lambda_j^2+1)}$. When $f\ge0$, it is a non-negative definite matrix. When $0>f>-\frac{1}{100n}$, for any $\xi_1, ... \xi_n\in\mathbb{R}$,
\[\begin{split}
&\sum_{i, j=1}^{n}-\frac{f\cos(\theta_0)}{\prod_{k=1}^{n}\sqrt{\lambda_k^2+1}}\frac{\xi_i\xi_j}{(\lambda_i^2+1)(\lambda_j^2+1)}\\
&\le\sum_{i, j=1}^{n}\frac{1}{100n}\cos(\theta_0)\frac{|\xi_i\xi_j|}{(\lambda_i^2+1)^{\frac{5}{4}}(\lambda_j^2+1)^{\frac{5}{4}}}\\
&=\frac{1}{100n}\cos(\theta_0)(\sum_{i=1}^{n}\frac{|\xi_i|}{(\lambda_i^2+1)^{\frac{5}{4}}})^2\\
&\le\frac{1}{100}\cos(\theta_0)\sum_{i=1}^{n}\frac{\xi_i^2}{(\lambda_i^2+1)^{\frac{5}{2}}}\\
&\le\frac{1}{100}\cos(\theta_0)\sum_{i=1}^{n}\sum_{j=1}^{n}\frac{\xi_i\xi_j\delta_{ij}}{\lambda_i(\lambda_i^2+1)^{2}}.
\end{split}\]
Since $\lambda_i>\cot(\theta_0)>1$ on $\Gamma$, it is easy to see that the second term is at most $-\frac{1}{4}$ times the first term. Similarly the third term and the fourth term are also at most $-\frac{1}{4}$ times the first term.

(3) It suffices to show that \[\frac{d}{dx}\frac{1}{x^2+1}+\frac{f}{\prod_{k=1}^n\sqrt{\lambda_k^2+1}}\frac{d}{dx}\frac{x}{x^2+1}=\frac{-2x}{(x^2+1)^2}+\frac{f}{\prod_{k=1}^n\sqrt{\lambda_k^2+1}}\frac{1-x^2}{(x^2+1)^2}\] is non-positive for all $x\in[\lambda_j,\lambda_i]$. When $f\ge0$, it is trivial. When $0>f>-\frac{1}{100n}$, it is at most \[\frac{-2x}{(x^2+1)^2}+\frac{1}{100n\lambda_i}\frac{x^2-1}{(x^2+1)^2}\le\frac{-2x}{(x^2+1)^2}+\frac{1}{x}\frac{x^2-1}{(x^2+1)^2}.\] This is indeed non-positive because $x\ge\lambda_j>1$.

(4)The point $\lambda_i=\cot(\frac{\theta_0}{n-1})$ belongs to $\partial\Gamma$ and $F$ at this point is negative because \[\frac{1}{100n}<\frac{(\cot^2(\frac{\theta_0}{n-1})+1)^{\frac{n}{2}}}{\cos(\theta_0)}\sin(\frac{\theta_0}{n-1})=\frac{1}{\cos(\theta_0)\sin^{\frac{n}{2}-1}(\frac{\theta_0}{n-1})}.\] Thus, it suffices to prove that if $F(f,\lambda)=0$, then $\lambda\not\in\partial\Gamma$. If $f\ge0$, it is obvious. If $0>f>-\frac{1}{100n}$, then $0\ge\theta_0-\sum_{i=1}^{n}\arctan(\frac{1}{\lambda_i})\ge-\frac{\theta_0}{n-1}\ge-\frac{\pi}{4}$. So $\lambda_i\not\in\partial\Gamma$ because
\[\frac{1}{100n}<(\inf_{x\in(1,\infty)}(\sqrt{x^2+1}\arctan(\frac{1}{x})))(\inf_{x\in[-\frac{\pi}{4},0]}\frac{\sin x}{x}).\]

(5) is obvious.
\end{proof}

Compared to Sz\'ekelyhidi's conditions in \cite{Szekelyhidi}, there are three major differences. First of all, $F$ also depends on $f$. Second of all, $\Gamma$ does not contain the positive orthant. Finally, even if we fix the $f$ variable, $F$ is only concave when $F=0$. However, we will show that his works still survive without much changes.

Proposition 5 of \cite{Szekelyhidi} only requires the concavity of $F$ when $F=0$. So it still holds. Sz\'ekelyhidi's $C^0$ estimate relies on the variant of Alexandroff-Bakelman-Pucci maximum principle similar to Lemma 9.2 of \cite{GilbargTrudinger}. Clearly it does not take derivatives of $f$. So Sz\'ekelyhidi's $C^0$ estimate is still true.

The next step is to prove that \[|\sqrt{-1}\partial\bar\partial\varphi|_\chi\le C_{5.2}(1+\sup_M|\nabla\varphi|_{\chi}^2).\] We will use the same notations as in \cite{Szekelyhidi} except that letter $f$ in \cite{Szekelyhidi} is replaced by $F$, the letter $F$ is replaced by $F_{\chi}$ and the letter $u$ is replaced by $\varphi$.  It is easy to see that (78) of \cite{Szekelyhidi} still holds. Now we differentiate the equation $F_\chi(f,\omega_\varphi)=0$. We see that
\[F_\chi^{ij}g_{i\bar j1}+F_\chi^f f_1=0,\] and
\[F_\chi^{pq,rs}g_{p\bar q1}g_{r\bar s\bar 1}+F_\chi^{kk}g_{k\bar k1\bar 1}+F_\chi^{kk,f}g_{k\bar k1}f_{\bar 1}+F_\chi^{kk,f}g_{k\bar k\bar 1}f_{1}+F_\chi^f f_{1\bar 1}=0\] because $F_\chi^{ff}=0$. Since $|F_\chi^{f}|=|-\frac{\cos(\theta_0)}{\prod_{i=1}^{n}\sqrt{\lambda_i^2+1}}|\le 1$, the term $F_\chi^f f_{1\bar 1}$ is  bounded. So the only additional term in (85) of \cite{Szekelyhidi} is $-C_0\lambda_1^{-1}|F_\chi^{kk,f}g_{k\bar k 1}|$ on the right hand side.
Instead of (94) of \cite{Szekelyhidi}, we get \[F_\chi^{kk}g_{k\bar kp}+F_\chi^f f_p=0.\]
Since $|F_\chi^f f_p|$ is bounded, the estimate in (95) still holds. So the only additional term in (99) and (104) of \cite{Szekelyhidi} is $-C_0\lambda_1^{-1}|F_\chi^{kk,f}g_{k\bar k 1}|$ on the right hand side. The case 1 in \cite{Szekelyhidi} will not happen. The additional term in (120) of \cite{Szekelyhidi} is also $-C_0\lambda_1^{-1}|F_\chi^{kk,f}g_{k\bar k 1}|$. However, recall that (67) of \cite{Szekelyhidi} is that
\[-F_\chi^{ij,rs}g_{i\bar j 1}g_{r\bar s\bar 1}\ge -F_{ij}g_{i\bar i 1}g_{j\bar j \bar 1}-\sum_{i>1}\frac{F_1-F_i}{\lambda_1-\lambda_i}|g_{i\bar 1 1}|.\]
(Remark that the letter $f$ in \cite{Szekelyhidi} is replaced by $F$ and the letter $F$ is replaced by $F_\chi$.) The term $-F_{ij}g_{i\bar i 1}g_{j\bar j \bar 1}$ was thrown away. However, this term is at least $\frac{\lambda_i}{2(\lambda_i^2+1)^2}|g_{i\bar i 1}|^2$. The term \[-C_0|F_\chi^{kk,f}g_{k\bar k 1}|=-C_0|\frac{f\cos(\theta_0)}{\prod_{k=1}^n\sqrt{\lambda_k^2+1}}\frac{\lambda_i}{\lambda_i^2+1}g_{i\bar i 1}|\ge-C_0|f|\frac{\lambda_i}{(\lambda_i^2+1)^{\frac{3}{2}}}|g_{i\bar i 1}|\] is at least $-\frac{\lambda_i}{2(\lambda_i^2+1)^2}|g_{i\bar i 1}|^2-C_{5.3}$. So Sz\'ekelyhidi's estimate \[|\sqrt{-1}\partial\bar\partial\varphi|_\chi\le C_{5.2}(1+\sup_M|\nabla\varphi|_{\chi}^2)\] still holds.

Sz\'ekelyhidi used the property that $\Gamma$ contains the positive orthant to prove the $C^2$ estimate \cite{Szekelyhidi}. We do not have this property. However, we can use Proposition 5.1 of \cite{CollinsJacobYau} to achieve this.

The Evans-Krylov estimate requires the uniform ellipticity and concavity of $F_\chi(f,.)$. Its relationship with the function $F$ was cited as (63) and (64) of \cite{Szekelyhidi}. By Proposition \ref{Property-of-F}, the conditions for the Evans-Krylov estimate are indeed true. The higher order estimate follows from standard elliptic theories. Finally, $\omega_\varphi$ will stay in the region $\Gamma_\chi$ along the continuity paths by Proposition \ref{Property-of-F} (4).

The analogy of Theorem \ref{Current-sub-solution} is the following:

\begin{theorem}
Fix a K\"ahler manifold $M^n$ with K\"ahler metrics $\chi$ and $\omega_0$. Suppose that for all $t>0$, there exist a constant $c_t>0$ and a smooth K\"ahler form $\omega_t\in[(1+t)\omega_0]$ satisfying $\omega_t\in\Gamma_\chi$, and \[\mathrm{Im}(\omega_\varphi+\sqrt{-1}\chi)^n+c_t\chi^n=\tan(\theta_0)\mathrm{Re}(\omega_\varphi+\sqrt{-1}\chi)^n.\]
Then there exist a constant $\epsilon_{5.4}>0$ and a current $\omega_{5.5}\in[\omega_0-\epsilon_{5.4}\chi]$ such that $\omega_{5.5}\in\bar\Gamma_\chi$ in the sense of current.
\label{Current-sub-solution-defomed-Hermitian-Yang-Mills}
\end{theorem}

The definition of a current being in $\bar\Gamma_\chi$ is similar to Definition \ref{Definition-current-sub-solution} except that we replace the condition \[c\omega^{n-1}-(n-1)\chi\wedge\omega^{n-2}\ge0\] by $\omega\in\bar\Gamma_\chi$ for K\"ahler form $\omega$. To simplify notations, for a K\"ahler form $\omega$, we define $P_\chi(\omega)$ and $Q_\chi(\omega)$ by
\[P_\chi(\omega)=\max_{j}(\sum_{i\not=j}\arctan(\frac{1}{\lambda_i})),\] and
\[Q_\chi(\omega)=\sum_{i}\arctan(\frac{1}{\lambda_i}),\]
where $\lambda_i$ are the eigenvalues of $\omega$ with respect to $\chi$. Then $\omega\in\bar\Gamma_\chi$ is equivalent to $P_\chi(\omega)\le\theta_0$.

The analogy of Lemma \ref{Subadditive-submatrix} is the following:
\begin{lemma}
Suppose that
\[\begin{bmatrix}
    A       & C \\
    \bar C^{T} & B
\end{bmatrix}>\begin{bmatrix}
    I   & O \\
    O & I
\end{bmatrix}.\]
Then
\[P_I(A-CB^{-1}\bar C^{T})+Q_I(B)\le P_I(\begin{bmatrix}
    A       & C \\
    \bar C^{T} & B
\end{bmatrix}).\]
\label{Subadditive-submatrix-defomed-Hermitian-Yang-Mills}
\end{lemma}
\begin{proof}
It is easy to see that $B>I$. Moreover, for any $\xi\not=0\in\mathbb{C}^n$, \[\begin{split}
\bar\xi^{T}(A-CB^{-1}\bar C^{T})\xi&=\begin{bmatrix}
    \bar\xi^{T}  & -\bar\xi^{T}CB^{-1} \\
\end{bmatrix}\begin{bmatrix}
    A       & C \\
    \bar C^{T} & B
\end{bmatrix}\begin{bmatrix}
    \xi        \\
    -B^{-1}\bar C^{T}\xi
\end{bmatrix}\\
&>|\xi|^2+|B^{-1}\bar C^{T}\xi|^2,
\end{split}\]
so $A-CB^{-1}\bar C^{T}>I$. Therefore, both hand sides of the inequality are well defined.

By restricting on the codimension 1 subspaces, it suffices to prove that
\[Q_I(A-CB^{-1}\bar C^{T})+Q_I(B)\le Q_I(\begin{bmatrix}
    A       & C \\
    \bar C^{T} & B
\end{bmatrix}).\]
For any $s\in[0,1]$, it is easy to see that
\[\begin{split}&\det\begin{bmatrix}
    A+\sqrt{-1}I   & sC \\
    \bar sC^{T} & B+\sqrt{-1}I
\end{bmatrix}\\
&=\det\begin{bmatrix}
    A-s^2C(B+\sqrt{-1}I)^{-1}\bar C^{T}+\sqrt{-1}I              & O \\
    O & B+\sqrt{-1}I
\end{bmatrix}.
\end{split}\]
So we need to compute
$\det(A-s^2C(B+\sqrt{-1}I)^{-1}\bar C^{T}+\sqrt{-1}I)$.

We already know that $B>I$, so \[\begin{split}
&(B+\sqrt{-1}I)^{-1}=B^{-1}(I+\sqrt{-1}B^{-1})^{-1}\\
&=\sum_{k=0}^{\infty}(-1)^{k}B^{-2k-1}-\sqrt{-1}\sum_{k=0}^{\infty}(-1)^{k}B^{-2k-2}.
\end{split}\]
Therefore \[\begin{split}
&A-s^2C(B+\sqrt{-1}I)^{-1}\bar C^{T}+\sqrt{-1}I\\
&=A-s^2(\sum_{k=0}^{\infty}(-1)^{k}CB^{-2k-1}\bar C^{T})+\sqrt{-1}(I+s^2\sum_{k=0}^{\infty}(-1)^{k}CB^{-2k-2}\bar C^{T}).
\end{split}\]
The real part is at least $A-s^2CB^{-1}\bar C^{T}>I$ and the imaginary part is also at least $I$. So \[\begin{split}&\arg\det\begin{bmatrix}
    A+\sqrt{-1}I   & sC \\
    \bar sC^{T} & B+\sqrt{-1}I
\end{bmatrix}\\
&=\arg\det(A-s^2C(B+\sqrt{-1}I)^{-1}\bar C^{T}+\sqrt{-1}I)+\arg\det(B+\sqrt{-1}I).
\end{split}\] if we define $\arg\det(X+\sqrt{-1}Y)$ as $Q_Y(X)$ for $X, Y>0$. In fact, this is true up to an integer times $2\pi$. However, both hand sides are continuous with respect to $s$ and this equation holds for $s=0$. So it holds for all $s\in[0,1]$.

Now it suffices to show that \[\arg\det(A-C(B+\sqrt{-1}I)^{-1}\bar C^{T}+\sqrt{-1}I)\ge\arg\det(A-CB^{-1}\bar C^{T}+\sqrt{-1}I).\] It follows from the facts that \[\sum_{k=0}^{\infty}(-1)^{k}CB^{-2k-2}\bar C^{T}\ge\sum_{k=1}^{\infty}(-1)^{k+1}CB^{-2k-1}\bar C^{T}\ge0\]
and \[A-CB^{-1}\bar C^{T}>I.\]
\end{proof}

Choose $C_{5.6}$ large enough such that $\theta_0+n\arctan(\frac{1}{C_{5.6}})<\frac{\pi}{4}$. The definitions of $\chi_{M\times M}$, $\chi_{M\times M,\epsilon_{5.7},\epsilon_{5.8}}$ and $f_{t,\epsilon_{5.7},\epsilon_{5.8}}$ are still the same as in Section 1. As before, there exists $\epsilon_{5.8}>0$ such that for $\epsilon_{5.7}$ small enough, $f_{t,\epsilon_{5.7},\epsilon_{5.8}}>-\frac{1}{200n}$.

Now we consider $\omega_{0,M\times M,t}=\pi_1^*\omega_t+C_{5.6}\pi_2^*\chi$. By Theorem  \ref{Estimate-defomed-Hermitian-Yang-Mills}, there exists $\omega_{t,\epsilon_{5.7},\epsilon_{5.8}}\in[\omega_{0,M\times M,t}]$ such that $P_{\chi_{M\times M}}(\omega_{t,\epsilon_{5.7},\epsilon_{5.8}})<\theta_0+n\arctan(\frac{1}{C_{5.6}})$ and
\[\begin{split}
&\mathrm{Im}(\omega_{t,\epsilon_{5.7},\epsilon_{5.8}}+\sqrt{-1}\chi_{M\times M})^{2n}+f_{t,\epsilon_{5.7},\epsilon_{5.8}}\chi_{M\times M}^{2n}\\
&=\tan(\theta_0+n\arctan(\frac{1}{C_{5.6}}))\mathrm{Re}(\omega_{t,\epsilon_{5.7},\epsilon_{5.8}}+\sqrt{-1}\chi_{M\times M})^{2n}.
\end{split}\]

Define $\omega_{1,t,\epsilon_{5.7},\epsilon_{5.8}}$ by \[\omega_{1,t,\epsilon_{5.7},\epsilon_{5.8}}=\frac{\sum_{k=0}^{\lfloor\frac{n-1}{2}\rfloor}(-1)^k\frac{n!}{(n-2k)!(2k+1)!}(\pi_1)_*(\omega_{t,\epsilon_{5.7},\epsilon_{5.8}}^{n-2k}\wedge\pi_2^*\chi^{2k+1})}{\int_M \mathrm{Im}(C_{5.6}\chi+\sqrt{-1}\chi)^n}.\] Fix $\epsilon_{5.8}$ and let $t$ and $\epsilon_{5.7}$ converge to 0. For small enough $\epsilon_{5.4}$, we shall expect $\omega_{5.5}$ to be the weak limit of $\omega_{1,t,\epsilon_{5.7},\epsilon_{5.8}}-\epsilon_{5.4}\chi$.

As before, we write $\omega_{t,\epsilon_{5.7},\epsilon_{5.8}}$ as \[\omega_{t,\epsilon_{5.7},\epsilon_{5.8}}=\omega^{(1)}_{t,\epsilon_{5.7},\epsilon_{5.8}}+\omega^{(2)}_{t,\epsilon_{5.7},\epsilon_{5.8}}+\omega^{(1,2)}_{t,\epsilon_{5.7},\epsilon_{5.8}}+\omega^{(2,1)}_{t,\epsilon_{5.7},\epsilon_{5.8}},\]
and assume that \[\pi_2^*\chi=\sqrt{-1}\sum_{i=1}^n dz^{(2)}_i\wedge d\bar z^{(2)}_i\] and \[\omega^{(2)}_{t,\epsilon_{5.7},\epsilon_{5.8}}=\sqrt{-1}\sum_{i=1}^n \lambda_i dz^{(2)}_i\wedge d\bar z^{(2)}_i\] at $(x_1,x_2)$.

To simplify notations, we omit $t,\epsilon_{5.7},\epsilon_{5.8}$, then
\[\omega_1=\frac{\sum_{k=0}^{\lfloor\frac{n-1}{2}\rfloor}(-1)^k(\pi_1)_*\hat\omega_{k}}{\int_M \mathrm{Im}(C_{5.6}\chi+\sqrt{-1}\chi)^n},\]
where $\hat\omega_{k}$ equals to
 \[\begin{split}
&\frac{n!}{(n-2k-1)!(2k+1)!}\omega^{(1)}\wedge(\omega^{(2)})^{n-2k-1}\wedge\pi_2^*\chi^{2k+1}\\
&+\frac{n!}{(n-2k-2)!(2k+1)!}\omega^{(1,2)}\wedge\omega^{(2,1)}\wedge(\omega^{(2)})^{n-2k-2}\wedge\pi_2^*\chi^{2k+1}\\
&=\sum_{i,j=1}^{n}\frac{\sqrt{-1}\omega^{(1)}_{i\bar j}dz_i^{(1)}\wedge d\bar z_j^{(1)}}{(2k+1)!}\wedge(\sum_{\alpha_1,...\alpha_{2k+1}\text{distinct}}\frac{1}{\lambda_{\alpha_1}...\lambda_{\alpha_{2k+1}}})(\omega^{(2)})^{n}\\
&-\sum_{i,j,l=1}^{n}\frac{\sqrt{-1}\omega^{(1,2)}_{i\bar l}\overline{\omega^{(1,2)}_{j\bar l}}dz_i^{(1)}\wedge d\bar z_j^{(1)}}{(2k+1)!}\wedge \sum_{\alpha_1,...\alpha_{2k+1},l \text{distinct}}\frac{1}{\lambda_l\lambda_{\alpha_1}...\lambda_{\alpha_{2k+1}}}(\omega^{(2)})^{n}.
\end{split}\]

Remark that \[\begin{split}
&\sum_{k=0}^{\lfloor\frac{n-1}{2}\rfloor}\frac{(-1)^k}{(2k+1)!}(\sum_{\alpha_1,...\alpha_{2k+1},l \text{distinct}}\frac{1}{\lambda_{\alpha_1}...\lambda_{\alpha_{2k+1}}}-\sum_{\alpha_1,...\alpha_{2k+1} \text{distinct}}\frac{1}{\lambda_{\alpha_1}...\lambda_{\alpha_{2k+1}}})\\
&=-\frac{1}{\lambda_l}\mathrm{Re}(1+\sqrt{-1}\frac{1}{\lambda_1})...(1+\sqrt{-1}\frac{1}{\lambda_{l-1}})(1+\sqrt{-1}\frac{1}{\lambda_{l+1}})...(1+\sqrt{-1}\frac{1}{\lambda_n})\\
&\le0,\end{split}\]
and
\[\begin{split}
&\sum_{k=0}^{\lfloor\frac{n-1}{2}\rfloor}\frac{(-1)^k}{(2k+1)!}\sum_{\alpha_1,...\alpha_{2k+1} \text{distinct}}\frac{1}{\lambda_{\alpha_1}...\lambda_{\alpha_{2k+1}}}(\omega^{(2)})^n\\
&=\mathrm{Im}(1+\sqrt{-1}\frac{1}{\lambda_1})...(1+\sqrt{-1}\frac{1}{\lambda_n})(\omega^{(2)})^n\\
&=\mathrm{Im}(\omega^{(2)}+\sqrt{-1}\pi_2^*\chi)^n.
\end{split}\]
By Lemma \ref{Subadditive-submatrix-defomed-Hermitian-Yang-Mills},
\[\begin{split}
&P_{\pi_1^*\chi}(\sqrt{-1}\sum_{i,j=1}^{n}((\omega^{(1)}_{i\bar j}-\sum_{k=1}^{n}\frac{1}{\lambda_k}\omega^{(1,2)}_{i\bar k}\overline{\omega^{(1,2)}_{j\bar k}})dz_i^{(1)}\wedge d\bar z_j^{(1)}))\\
&\le P_{\chi_{M\times M}}(\omega)-Q_{\pi_2^*\chi}(\omega^{(2)})\\
&<\theta_0+n\arctan(\frac{1}{C_{5.6}})-Q_{\pi_2^*\chi}(\omega^{(2)}).
\end{split}\]

So by the monotonicity and convexity of $P_\chi$, \[\begin{split}
&P_\chi(\omega_1)\\
&<\frac{\int_{\{x_1\}\times M}(\theta_0+n\arctan(\frac{1}{C_{5.6}})-Q_{\pi_2^*\chi}(\omega^{(2)}))\mathrm{Im}(\omega^{(2)}+\sqrt{-1}\pi_2^*\chi)^n}{\int_M \mathrm{Im}(C_{5.6}\chi+\sqrt{-1}\chi)^n}\\
&=\theta_0+n\arctan(\frac{1}{C_{5.6}})-\frac{\int_{\{x_1\}\times M}(Q_{\pi_2^*\chi}(\omega^{(2)}))\mathrm{Im}(\omega^{(2)}+\sqrt{-1}\pi_2^*\chi)^n}{\int_M \mathrm{Im}(C_{5.6}\chi+\sqrt{-1}\chi)^n}.
\end{split}\]
Using the convexity of $x\arctan x$ and the fact that \[Q_{\pi_2^*\chi}(\omega^{(2)})=\arctan(\frac{\mathrm{Im}(\omega^{(2)}+\sqrt{-1}\pi_2^*\chi)^n}{\mathrm{Re}(\omega^{(2)}+\sqrt{-1}\pi_2^*\chi)^n}),\]
it is easy to see that the minimum of \[\int_{\{x_1\}\times M}(Q_{\pi_2^*\chi}(\omega^{(2)}))\mathrm{Im}(\omega^{(2)}+\sqrt{-1}\pi_2^*\chi)^n\] is achieved if $Q_{\pi_2^*\chi}(\omega^{(2)})$ is a constant, which must be $n\arctan(\frac{1}{C_{5.6}})$. Thus $P_\chi(\omega_1)<\theta_0$. Back to our original notations, it means that $P_\chi(\omega_{1,t,\epsilon_{5.7},\epsilon_{5.8}})<\theta_0$.

The rest part of Section 3 still holds because $\omega_{t,\epsilon_{5.7},\epsilon_{5.8}}^{n-k}$ has no concentration of mass on the diagonal if $k\ge1$. Most part of Section 4 also holds because $\bar\Gamma_\chi$ is convex. We only need to prove the following analogy of Lemma \ref{Stability-on-blow-up}:

\begin{lemma}
Let $C_{5.9}=\cot(\frac{\epsilon_{1.1}}{6n})$. Then for all small enough $t$ and all $q$-dimensional subvarieties $V$ of $\tilde M$, as long as $q<n$, \[\mathrm{Im}\int_V e^{\sqrt{-1}(\frac{\epsilon_{1.1}}{3}-\theta_0)}((1+C_{5.9}t)\pi^*\omega_0+C_{5.9}t^2\omega_{4.13}+\sqrt{-1}(\pi^*\chi+t^2\omega_{4.13}))^q\le0.\]
\label{Stability-on-blow-up-deformed-Hermitian-Yang-Mills}
\end{lemma}
\begin{proof}
We already know that
\[\mathrm{Im}\int_V e^{\sqrt{-1}(\frac{\epsilon_{1.1}}{3}-\theta_0)}((1+C_{5.9}t)\pi^*\omega_0+\sqrt{-1}(1+C_{5.9}t)\pi^*\chi)^q\le0.\]

So it suffices to show that
\[\begin{split}
&\mathrm{Im}\int_V e^{\sqrt{-1}(\frac{\epsilon_{1.1}}{3}-\theta_0)}((1+C_{5.9}t)\pi^*\omega_0+C_{5.9}t^2\omega_{4.13}+\sqrt{-1}(\pi^*\chi+t^2\omega_{4.13}))^q\\
&\le\mathrm{Im}\int_V e^{\sqrt{-1}(\frac{\epsilon_{1.1}}{3}-\theta_0)}((1+C_{5.9}t)\pi^*\omega_0+\sqrt{-1}(1+C_{5.9}t)\pi^*\chi)^q.
\end{split}\]

As before, there exists a smooth function $\varphi_{5.10}$ on a neighborhood $O_{5.11}$ of $\pi(E)$ in $M$ such that $\omega_{5.10}=\omega_0+\sqrt{-1}\partial\bar\partial\varphi_{5.10}>0$ satisfies $P_{\chi}(\omega_{5.10})<\theta_0-\frac{\epsilon_{1.1}}{2}$ on $O_{5.11}$. We define $\varphi_{4.17}$ as before on $M\setminus\pi(E)$. For any $s>0$, by our assumption, there exists $\varphi_{5.12}$ such that $\omega_{5.12}=(1+s)\omega_0+\sqrt{-1}\partial\bar\partial\varphi_{5.12}>0$ satisfies $P_\chi(\omega_{5.12})<\theta_0<\frac{\pi}{4}$. Choose $s$ small enough such that $P_\chi(\frac{\omega_{5.12}}{1+s})<\frac{\pi}{4}$. Choose a large enough constant $C_{5.13}$. Let $\varphi_{5.14}$ be the regularized maximum of $\varphi_{5.10}$ and $\frac{\omega_{5.12}}{1+s}+C_{5.13}^{-1}\varphi_{4.17}+C_{5.13}$. Then it is easy to see that $\varphi_{5.14}$ is smooth and $P_\chi(\omega_{5.14})<\frac{\pi}{4}$ on $M$ for $\omega_{5.14}=\omega_0+\sqrt{-1}\partial\bar\partial\varphi_{5.14}>0$. Moreover, there exists a smaller neighborhood $O_{5.15}$ of $\pi(E)$ such that $\varphi_{5.14}=\varphi_{5.10}$ on $O_{5.15}\subset O_{5.11}$.

After replacing $\omega_0$ by $\omega_{5.14}$, it suffices to show that
\[\begin{split}
&\mathrm{Im}(e^{\sqrt{-1}(\frac{\epsilon_{1.1}}{3}-\theta_0)}((1+C_{5.9}t)\pi^*\omega_{5.14}+C_{5.9}t^2\omega_{4.13}+\sqrt{-1}(\pi^*\chi+t^2\omega_{4.13}))^q)\\
&\le\mathrm{Im}(e^{\sqrt{-1}(\frac{\epsilon_{1.1}}{3}-\theta_0)}((1+C_{5.9}t)\pi^*\omega_{5.14}+\sqrt{-1}(1+C_{5.9}t)\pi^*\chi)^q).
\end{split}\]
First of all, \[\begin{split}
&\mathrm{Im}(e^{\sqrt{-1}(\frac{\epsilon_{1.1}}{3}-\theta_0)}((1+C_{5.9}t)\pi^*\omega_{5.14}+\sqrt{-1}(1+C_{5.9}t)\pi^*\chi)^q)\\
&-\mathrm{Im}(e^{\sqrt{-1}(\frac{\epsilon_{1.1}}{3}-\theta_0)}((1+C_{5.9}t)\pi^*\omega_{5.14}+\sqrt{-1}\pi^*\chi)^q)\\
&=\int_{1}^{1+C_{5.9}t}\mathrm{Re}(q e^{\sqrt{-1}(\frac{\epsilon_{1.1}}{3}-\theta_0)}((1+C_{5.9}t)\pi^*\omega_{5.14}+\sqrt{-1}\tau\pi^*\chi)^{q-1}\wedge\pi^*\chi)d\tau\\
&\ge0
\end{split}\]
using a calculation similar to Lemma 8.2 of \cite{CollinsJacobYau} because $P_\chi(\omega_{5.14})<\frac{\pi}{4}$ on $M$.
If the point is outside $\pi^{-1}(O_{5.15})$, then as before, we get the required inequality if $t$ is small enough. If the point is inside $\pi^{-1}(O_{5.15})$, then \[\begin{split}
&\mathrm{Im}(e^{\sqrt{-1}(\frac{\epsilon_{1.1}}{3}-\theta_0)}((1+C_{5.9}t)\pi^*\omega_{5.14}+C_{5.9}t^2\omega_{4.13}+\sqrt{-1}(\pi^*\chi+t^2\omega_{4.13}))^q)\\
-&\mathrm{Im}(e^{\sqrt{-1}(\frac{\epsilon_{1.1}}{3}-\theta_0)}((1+C_{5.9}t)\pi^*\omega_{5.14}+\sqrt{-1}\pi^*\chi)^q)\\
=&\mathrm{Im}(e^{\sqrt{-1}(\frac{\epsilon_{1.1}}{3}-\theta_0)}\sum_{k=0}^{q-1}\frac{q!}{k!(q-k)!}((1+C_{5.9}t)\pi^*\omega_{5.14}+\sqrt{-1}\pi^*\chi)^k\\\
&\wedge(C_{5.9}t^2\omega_{4.13}+\sqrt{-1}t^2\omega_{4.13})^{q-k})\\
\le&0
\end{split}\]
using a calculation similar to Lemma 8.2 of \cite{CollinsJacobYau} because $P_\chi(\omega_{5.14})<\theta_0-\frac{\epsilon_{1.1}}{2}$ on $O_{5.15}$. We are done.
\end{proof}


\begin{thebibliography}{99}
\bibitem{ApostolovCalderbankGauduchonTonnesen} \textit{Vestislav~Apostolov}, \textit{David~M.~J.~Calderbank}, \textit{Paul~Gauduchon}, and \textit{Christina~W.~T\o nnesen-Friedman}, Hamiltonian 2-forms in K\"ahler geometry III: extremal metrics and stability,  Invent. Math., \textbf{173} (2008), no. 3, 547--601.
\bibitem{Aubin} \textit{Thierry~Aubin}, R\'eduction du cas positif de l'\'equation de Monge-Amp\`ere sur les vari\'et\'es k\"ahl\'eriennes compactes \`a la d\'emonstration d'une in\'egalit\'e, J. Funct. Anal., \textbf{57} (1984), no. 2, 143--153.
 \bibitem{BedfordTaylor} \textit{Eric~Bedford} and \textit{B.~A.~Taylor}, The Dirichlet problem for a complex Monge-Amp\`ere equation, Invent. Math., \textbf{37} (1976), no. 1, 1--44.
 \bibitem{BlockiKolodziej} \textit{Zbigniew~B\l{}ocki} and \textit{S\l{}awomir~Ko\l{}odziej}, Regularization of plurisubharmonic functions on manifolds, Proc. Amer. Math. Soc., \textbf{135} (2007), no. 7, 2089--2093 .
 \bibitem{Chen} \textit{Xiuxiong~Chen}, On the Lower Bound of the Mabuchi Energy and Its Application, Internat. Math. Res. Notices, 2000, no. 12, 607--623.
 \bibitem{ChenCheng1} \textit{Xiuxiong~Chen} and \textit{Jingrui~Cheng}, On the constant scalar curvature K\"ahler metrics, apriori estimates, preprint, \url{https://arxiv.org/abs/1712.06697}
 \bibitem{ChenCheng2} \textit{Xiuxiong~Chen} and \textit{Jingrui~Cheng}, On the constant scalar curvature K\"ahler metrics, existence results, preprint, \url{https://arxiv.org/abs/1801.00656}
 \bibitem{ChenCheng3} \textit{Xiuxiong~Chen} and \textit{Jingrui~Cheng}, On the constant scalar curvature K\"ahler metrics, general automorphism group, preprint, \url{https://arxiv.org/abs/1801.05907}
 \bibitem{ChenDonaldsonSun1} \textit{Xiuxiong~Chen}, \textit{Simon~Donaldson} and \textit{Song~Sun}, K\"ahler-Einstein metrics on Fano manifolds. I: Approximation of metrics with cone singularities, J. Amer. Math. Soc., \textbf{28} (2015), no.1, 183--197.
 \bibitem{ChenDonaldsonSun2} \textit{Xiuxiong~Chen}, \textit{Simon~Donaldson} and \textit{Song~Sun}, K\"ahler-Einstein metrics on Fano manifolds. II: Limits with cone angle less than 2$\pi$, J. Amer. Math. Soc., \textbf{28} (2015), no.1, 199--234.
 \bibitem{ChenDonaldsonSun3} \textit{Xiuxiong~Chen}, \textit{Simon~Donaldson} and \textit{Song~Sun}, K\"ahler-Einstein metrics on Fano manifolds. III: Limits as cone angle approaches 2$\pi$ and completion of the main proof, J. Amer. Math. Soc., \textbf{28} (2015), no.1, 235--278.
 \bibitem{CollinsJacobYau} \textit{Tristan~C.~Collins}, \textit{Adam~Jacob} and \textit{Shing-Tung~Yau}, (1,1) forms with specified Lagrangian phase: a priori estimates and algebraic obstructions, preprint, \url{https://arxiv.org/abs/1508.01934}
 \bibitem{CollinsSzekelyhidi} \textit{Tristan~C.~Collins} and \textit{G\'abor~Sz\'ekelyhidi}, Convergence of the J-flow on toric manifolds, J. Differential Geom., \textbf{107} (2017), no. 1, 47--81.
 \bibitem{CollinsXieYau} \textit{Tristan~C.~Collins}, \textit{Dan~Xie}, and \textit{Shing-Tung~Yau}, The deformed Hermitian-Yang-Mills equation in geometry and physics, preprint, \url{https://arxiv.org/abs/1712.00893}
 \bibitem{CollinsYau} \textit{Tristan~C.~Collins} and \textit{Shing-Tung~Yau}, Moment maps, nonlinear PDE, and stability in mirror symmetry, preprint, \url{https://arxiv.org/abs/1811.04824}
 \bibitem{Darvas} \textit{Tam\'as~Darvas}, The Mabuchi geometry of finite energy classes, Adv. Math., \textbf{285} (2015), 182--219.
 \bibitem{Demailly} \textit{Jean-Pierre~Demailly}, Complex Analytic and Differential Geometry, preprint, \url{https://www-fourier.ujf-grenoble.fr/~demailly/manuscripts/agbook.pdf}
 \bibitem{DemaillyPaun} \textit{Jean-Pierre~Demailly} and \textit{Mihai~Paun}, Numerical characterization of the K\"ahler cone of a compact K\"ahler manifold, Ann. of Math., (2) \textbf{159} (2004), no. 3, 1247--1274.
 \bibitem{DervanRoss} \textit{Ruadha\'{i}~Dervan} and \textit{Julius~Ross}, K-stability for K\"ahler manifolds, Math. Res. Lett., \textbf{24} (2017), no. 3, 689--739.
\bibitem{DonaldsonGauge} \textit{Simon~K.~Donaldson}, Anti self-dual Yang-Mills connections over complex algebraic surfaces and stable vector bundles, Proceedings of the London Mathematical Society, (3) \textbf{50} (1985), 1--26.
 \bibitem{DonaldsonJEquation} \textit{S.~K.~Donaldson}, Moment maps and diffeomorphisms, Sir Michael Atiyah: a great mathematician of the twentieth century, Asian J. Math., \textbf{3} (1999), no. 1, 1--15.
 \bibitem{Donaldson} \textit{S.~K.~Donaldson},  Scalar Curvature and Stability of Toric Varieties, J. Differential Geom., \textbf{62} (2002), no. 2, 289--349.
 \bibitem{Evans1} \textit{Lawrence~C.~Evans}, Classical solutions of fully nonlinear, convex, second order elliptic equations, Comm. Pure Appl. Math., \textbf{25} (1982), 333--363.
 \bibitem{Evans2} \textit{Lawrence~C.~Evans}, Classical solutions of the Hamilton-Jacobi Bellman equation for uniformly elliptic operators, Trans. Amer. Math. Soc., \textbf{275} (1983), 245--255.
 \bibitem{GilbargTrudinger} \textit{David~Gilbarg} and \textit{Neil~S.~Trudinger}, Elliptic partial differential equations of second order. Reprint of the 1998 edition, Springer-Verlag, Berlin, 2001.
 \bibitem{JacobYau} \textit{Adam~Jacob} and \textit{Shing-Tung~Yau}, A special Lagrangian type equation for holomorphic line bundles,
Math. Ann., \textbf{369} (2017), no. 1--2, 869--898.
 \bibitem{Krylov} \textit{N.~V.~Krylov}, Boundedly nonhomogeneous elliptic and parabolic equations, Izvestia Akad. Nauk. SSSR, \textbf{46} (1982), 487--523. English translation in Math. USSR Izv., \textbf{20} (1983), No. 3, 459--492.
 \bibitem{LejmiSzekelyhidi} \textit{Mehdi~Lejmi} and \textit{Gábor~Sz\'ekelyhidi}, The J-flow and stability, Adv. Math., \textbf{274} (2015), 404--431.
 \bibitem{LeungYauZaslow} \textit{Naichung~Conan~Leung}, \textit{Shing-Tung~Yau}, and \textit{Eric~Zaslow}, From special Lagrangian to Hermitian-Yang-Mills via Fourier-Mukai, Adv. Theor. Math. Phys., \textbf{4} (2000), no. 6, 1319--1341.
 \bibitem{Pingali} \textit{Vamsi~Pingali}, A note on the deformed Hermitian Yang-Mills PDE, Complex Var. Elliptic Equ., \textbf{64} (2019), no. 3, 503--518.
 \bibitem{RossThomas} \textit{Julius~Ross} and \textit{Richard~Thomas}, An obstruction to the existence of constant scalar curvature K\"ahler metrics,  J. Differential Geom., \textbf{72} (2006), no. 3, 429--466.
 \bibitem{Siu} \textit{Yum-Tong~Siu}, Analyticity of sets associated to Lelong numbers and the extension of closed positive currents, Invent. Math., \textbf{27} (1974), 53--156.
 \bibitem{SjostromDyrefelt} \textit{Zakaria~Sj\"ostr\"om~Dyrefelt}, K-semistability of cscK manifolds with transcendental cohomology class, J. Geom. Anal., \textbf{28} (2018), no. 4, 2927--2960.
 \bibitem{SongWeinkove} \textit{Jian~Song} and \textit{Ben~Weinkove}, On the convergence and singularities of the J-flow with applications to the Mabuchi energy, Comm. Pure Appl. Math., \textbf{61} (2008), no. 2, 210--229.
 \bibitem{StromingerYauZaslow} \textit{Andrew~Strominger}, \textit{Shing-Tung~Yau}, and \textit{Eric~Zaslow}, Mirror Symmetry is T-duality, Nucl. Phys., \textbf{B479} (1996), 243--259.
 \bibitem{Szekelyhidi} \textit{G\'abor~Sz\'ekelyhidi}, Fully non-linear elliptic equations on compact Hermitian manifolds, J. Differential Geom., \textbf{109} (2018), no. 2, 337--378.
 \bibitem{Tian} \textit{Gang~Tian}, K\"ahler-Einstein metrics with positive scalar curvature, Inv. Math., \textbf{130} (1997) 1--37.
 \bibitem{Trudinger} \textit{Neil~S.~Trudinger}, Fully nonlinear, uniformly elliptic equations under natural structure conditions, Trans. Amer. Math. Soc., \textbf{278} (1983), 751--769.
 \bibitem{UhlenbeckYau} \textit{Karen~Uhlenbeck} and \textit{Shing-Tung~Yau}, On the existence of Hermitian–Yang-Mills connections in stable vector bundles. Frontiers of the mathematical sciences: 1985 (New York, 1985). Communications on Pure and Applied Mathematics 39 (1986), no. S, suppl., S257–S293.
 \bibitem{Yau} \textit{Shing-Tung~Yau}, On the Ricci curvature of a compact K\"ahler manifold and the complex Monge-Amp\`ere equation. I,  Comm. Pure Appl. Math., \textbf{31} (1978), no. 3, 339--411.
 \bibitem{YauStability} \textit{Shing-Tung~Yau}, Review of K\"ahler-Einstein metrics in algebraic geometry, Israel Math. Conference Proceedings, Bar Ilan Univ., 433--443, 1996.
 \bibitem{Zheng} \textit{Kai~Zheng}, I-properness of Mabuchi's K-energy, Calc. Var. Partial Differential Equations, \textbf{54} (2015), no. 3, 2807--2830.
 \end{thebibliography}
\end{document}